\documentclass[12pt,reqno,a4paper]{amsart}

\usepackage[a4paper,margin=4cm]{geometry}
\usepackage[T1]{fontenc}
\usepackage{lmodern}
\usepackage{microtype}
\usepackage{amsmath,amssymb,amsthm}
\usepackage{aliascnt}
\usepackage{enumitem}

\usepackage{xcolor}
\usepackage[
  colorlinks=true,
  citecolor=blue!75!black,
  linkcolor=red!65!black,
  linktocpage=true
]{hyperref}

\usepackage[nameinlink,noabbrev]{cleveref}

\hypersetup{
  pdftitle={Negative Effective Divisors and Bridgeland Stability of Line Bundles on Surfaces},
  pdfauthor={Anthony Mäkelä},
  pdfsubject={Bridgeland stability of line bundles on smooth projective surfaces},
  pdfkeywords={Bridgeland stability, line bundles, projective surfaces, Harder--Narasimhan filtrations, Lorentzian geometry}
}
\setlist[enumerate]{leftmargin=2.25em}
\setlist[itemize]{leftmargin=2.25em}
\allowdisplaybreaks
\numberwithin{equation}{section}

\newtheorem{theorem}{Theorem}[section]
\newaliascnt{proposition}{theorem}
\newtheorem{proposition}[proposition]{Proposition}
\aliascntresetthe{proposition}
\newaliascnt{lemma}{theorem}
\newtheorem{lemma}[lemma]{Lemma}
\aliascntresetthe{lemma}
\newaliascnt{corollary}{theorem}
\newtheorem{corollary}[corollary]{Corollary}
\aliascntresetthe{corollary}
\theoremstyle{definition}
\newaliascnt{definition}{theorem}

\aliascntresetthe{definition}
\newaliascnt{example}{theorem}
\newtheorem{example}[example]{Example}
\aliascntresetthe{example}
\theoremstyle{remark}
\newaliascnt{remark}{theorem}
\newtheorem{remark}[remark]{Remark}
\aliascntresetthe{remark}

\crefname{theorem}{theorem}{theorems}
\Crefname{theorem}{Theorem}{Theorems}
\crefname{proposition}{proposition}{propositions}
\Crefname{proposition}{Proposition}{Propositions}
\crefname{lemma}{lemma}{lemmas}
\Crefname{lemma}{Lemma}{Lemmas}
\crefname{corollary}{corollary}{corollaries}
\Crefname{corollary}{Corollary}{Corollaries}
\crefname{definition}{definition}{definitions}
\Crefname{definition}{Definition}{Definitions}
\crefname{example}{example}{examples}
\Crefname{example}{Example}{Examples}
\crefname{remark}{remark}{remarks}
\Crefname{remark}{Remark}{Remarks}

\newcommand{\C}{\mathbb C}
\newcommand{\R}{\mathbb R}
\newcommand{\NS}{\operatorname{NS}}
\newcommand{\Coh}{\operatorname{Coh}}
\newcommand{\Db}{\mathrm D^{\mathrm b}}
\newcommand{\RSheafHom}{\mathbf R\mathcal{H}\!om_X}
\newcommand{\Ext}{\operatorname{Ext}}
\newcommand{\ch}{\operatorname{ch}}
\newcommand{\rk}{\operatorname{rk}}
\newcommand{\Arg}{\operatorname{arg}}
\newcommand{\im}{\operatorname{Im}}
\newcommand{\cO}{\mathcal O}
\newcommand{\cA}{\mathcal A}
\newcommand{\cT}{\mathcal T}
\newcommand{\cF}{\mathcal F}
\newcommand{\eps}{\varepsilon}
\newcommand{\wt}{\widetilde}
\newcommand{\ol}{\overline}
\newcommand{\abs}[1]{\lvert #1\rvert}
\newcommand{\norm}[1]{\lVert #1\rVert}

\title[Negative Effective Divisors and Bridgeland Stability]{Negative Effective Divisors and Bridgeland Stability of Line Bundles on Surfaces}
\author{Anthony Mäkelä}
\address{Mathematical Sciences, University of Gothenburg, Sweden}
\email{anthony.makela@gu.se}
\date{}

\begin{document}
\raggedbottom

\begin{abstract}
Let $X$ be a connected smooth complex projective surface. We prove an
effective-divisor version of the
Arcara--Miles conjecture, together with its strict analogue. For every divisorial Bridgeland stability condition, failure of stability,
respectively semistability, of a line bundle or its relevant shift is detected
by a natural subobject associated with a non-zero effective Cartier divisor
$C$ satisfying $C^2<0$.

The proof combines minimal-rank destabilizers, slope Harder--Narasimhan
filtrations, the Bogomolov--Gieseker inequality, and an ordered Lorentzian
partial-sum estimate that forces the minimal rank to be one. Consequently, strict semistability of a line bundle or its relevant shift is detected by a negative effective divisor, and a numerical semistability condition arising from the twisted deformed Hermitian--Yang--Mills equation is equivalent to stability under all integral scalings.
\end{abstract}

\maketitle

\section{Introduction}\label{sec:introduction}

Let $X$ be a connected smooth complex projective surface. For
$B,\omega\in\NS(X)_\R$, with $\omega$ ample, consider the standard
divisorial stability condition
\begin{equation}\label{eq:central-charge-integral}
  \sigma_{B,\omega}=(Z_{B,\omega},\cA_{B,\omega}),
  \qquad
  Z_{B,\omega}(E)
  :=-\int_X e^{-(B+i\omega)}\ch(E).
\end{equation}
Here $\cA_{B,\omega}$ is the heart obtained by tilting $\Coh(X)$ at the
torsion pair determined by the $\omega$-slope threshold $B\cdot\omega$. The precise definition of $\cA_{B,\omega}$ is recalled in \Cref{sec:preliminaries}. The general notion of a stability condition is due to Bridgeland
\cite{Bridgeland}. This construction originated for $K3$ surfaces in
\cite[Section~6]{BriK3} and was extended to arbitrary smooth projective
surfaces in \cite[Section~2]{AB}. See also
\cite[Section~6]{MacriSchmidt} for a modern account, including the support
property. We refer to members of this family as divisorial
stability conditions.
Arguments of non-zero central charges are taken in $(0,\pi]$. Thus a non-zero
object $T\in\cA_{B,\omega}$ is stable, respectively semistable, if every
non-zero proper subobject $E\subset T$ satisfies
\begin{align*}
  \Arg Z_{B,\omega}(E)&<\Arg Z_{B,\omega}(T)
    &&\text{in the stable case},\\
  \Arg Z_{B,\omega}(E)&\leq\Arg Z_{B,\omega}(T)
    &&\text{in the semistable case}.
\end{align*}
We use the usual shift-invariant meaning of stability for arbitrary objects of
$\Db(X)$. If $L$ is a line bundle, put
\begin{equation*}
  s_{B,\omega}(L):=\omega\cdot(c_1(L)-B).
\end{equation*}
Then $L\in\cA_{B,\omega}$ when $s_{B,\omega}(L)>0$, while
$L[1]\in\cA_{B,\omega}$ when $s_{B,\omega}(L)\leq0$.

Arcara and Miles conjectured that the ``only objects that could destabilize''
a line bundle or its relevant shift are the natural objects associated with
curves of negative self-intersection \cite[Conjecture~1]{AM}. They use the term \emph{destabilizes} for a strict phase inequality \cite[Sections~2--3]{AM}, although their later arguments also involve containment and comparison of walls. Fan formulated failure of stability as a
detection statement: a divisorial subobject associated with a curve of
negative self-intersection has phase at least that of the relevant object
\cite[Conjecture~2.8]{Fan}. Stoppa considered the specialization $B=0$ with a Kähler class. In the
conventions fixed above, his formulation concerns failure of semistability
\cite[Conjecture~6.1 and Section~6.1.1]{Stoppa}. For the del Pezzo surface of
Picard rank three, Mizuno and Yoshida proved the stronger statement that the
divisorial subobject can be chosen with phase at least that of a prescribed
destabilizing object
\cite[Conjecture~A, Theorem~3.2, and Propositions~3.14 and~3.15]{MY}.

Let $T$ denote whichever of $L$ and $L[1]$ belongs to the tilted heart.
For a non-zero effective Cartier divisor $C$ with self-intersection number $C^2<0$, let $D_C$ denote the corresponding natural subobject
\[
  L(-C)\subset L
  \qquad\text{or}\qquad
  L(C)|_C\subset L[1],
\]
whenever the corresponding morphism is a monomorphism in the tilted heart.

Our theorem proves the two implications
\[
\begin{aligned}
 T\ \text{not stable}
   &\Longrightarrow \Arg Z(D_C)\geq\Arg Z(T),\\
 T\ \text{not semistable}
   &\Longrightarrow \Arg Z(D_C)>\Arg Z(T)
\end{aligned}
\]
for some $C$. This is a detection result rather than a classification. $D_C$ need not coincide with a given destabilizing subobject, and its phase is compared only with that of $T$.
The distinction between detection and classification is essential as a strict destabilizing subobject need not itself be a line bundle, as shown in \Cref{ex:non-line-bundle-destabilizer}.

We formulate our results for non-zero effective Cartier divisors, which are
not assumed to be reduced or irreducible. We use the term \emph{integral
curve} when reducedness and irreducibility are required.

For an effective Cartier divisor $C$, write
\[
  L(C)|_C:=L\otimes\cO_X(C)\otimes\cO_C,
\]
and let $L(C)|_C\to L[1]$ denote the connecting morphism associated with
\[
  0\longrightarrow L\longrightarrow L(C)\longrightarrow L(C)|_C
  \longrightarrow0.
\]

\begin{theorem}
\label{thm:main}
Let $X$ be a connected smooth complex projective surface, let
$B,\omega\in\NS(X)_\R$ with $\omega$ ample, let $L$ be a line bundle, and
put $s=s_{B,\omega}(L)$.
\begin{enumerate}[label=\textup{(\alph*)}]
  \item Suppose $s>0$, so that $L\in\cA_{B,\omega}$.
  \begin{enumerate}[label=\textup{(\roman*)}]
    \item The object $L$ is not $\sigma_{B,\omega}$-stable if and only if
    there is a non-zero effective Cartier divisor $C$ such that
    \[
      C^2<0,\qquad s-\omega\cdot C>0,
    \]
    and
    \begin{equation}\label{eq:main-unshifted-phase}
        \Arg Z_{B,\omega}(L(-C))
        \geq
        \Arg Z_{B,\omega}(L).
    \end{equation}
    \item The object $L$ is not $\sigma_{B,\omega}$-semistable if and only
    if there is such a divisor $C$ for which
    \eqref{eq:main-unshifted-phase} is strict.
  \end{enumerate}

  \item Suppose $s\leq0$, so that $L[1]\in\cA_{B,\omega}$.
  \begin{enumerate}[label=\textup{(\roman*)}]
    \item The object $L[1]$ is not $\sigma_{B,\omega}$-stable if and only if
    there is a non-zero effective Cartier divisor $C$ such that
    \[
      C^2<0,\qquad s+\omega\cdot C\leq0,
    \]
    and
    \begin{equation}\label{eq:main-shifted-phase}
        \Arg Z_{B,\omega}(L(C)|_C)
        \geq
        \Arg Z_{B,\omega}(L[1]).
    \end{equation}
    \item The object $L[1]$ is not $\sigma_{B,\omega}$-semistable if and only
    if there is such a divisor $C$ for which
    \eqref{eq:main-shifted-phase} is strict.
  \end{enumerate}
\end{enumerate}
\end{theorem}
By \Cref{lem:divisorial-subobject-criterion}, the two slope inequalities
in \Cref{thm:main} are exactly the conditions under which the
corresponding natural morphisms $L(-C)\to L$ and
$L(C)|_C\to L[1]$ are monomorphisms in the tilted heart.

Only the forward implications require proof. Starting from a weak or strict
destabilizer of minimal positive rank, we construct the required negative
divisorial subobject.

\begin{remark}
The cited formulations use the unqualified term \emph{curve}.
Arcara--Miles write the divisorial part of a rank-one subsheaf as
$I_Z(-C)$ for an ``effective curve'' $C$, while using ``irreducible curve''
when irreducibility is intended
\cite[Lemma~4.1 and Theorem~1.1]{AM}. Fan and Stoppa likewise retain the
unqualified terminology in \cite[Conjecture~2.8]{Fan} and
\cite[Conjecture~6.1]{Stoppa}. Accordingly, \Cref{thm:main} proves the
effective-divisor form of these conjectures.

If \emph{curve} is instead required to mean an integral curve, this is a stronger assertion than our theorem covers. Indeed, if a detecting effective
divisor has a decomposition
\[
  C=\sum_i m_iC_i
\]
into integral curves, then $C^2<0$ implies that some component $C_i$ has
$C_i^2<0$, since distinct integral curves have non-negative intersection.
However, the phase inequality need not pass to such a component. For instance, in the notation of \Cref{sec:phase-functional},
\[
  \Lambda\bigl(\cO_X(-C_1-C_2)\bigr)
  =
  \Lambda\bigl(\cO_X(-C_1)\bigr)
  +\Lambda\bigl(\cO_X(-C_2)\bigr)
  +C_1\cdot C_2.
\]
Since $C_1\cdot C_2\geq0$, non-negativity of the left-hand side does not
imply non-negativity of either componentwise term. Obtaining an integral
detecting curve would therefore require an additional argument.
\end{remark}

We next record three consequences of \Cref{thm:main}. The first says that
strict semistability of a line bundle, or of its relevant shift, is detected
by a negative effective divisor.

\begin{corollary}
\label{cor:actual-wall}
Let $L$ be a line bundle, and let $T\in\{L,L[1]\}$ be the shift belonging to
$\cA_{B,\omega}$. If $T$ is semistable but not stable, then there is a
non-zero effective Cartier divisor $C$ with $C^2<0$ and a short exact sequence
\[
  0\longrightarrow D_C\longrightarrow T\longrightarrow Q_C\longrightarrow0
\]
in $\cA_{B,\omega}$ such that $D_C$, $T$, and $Q_C$ are semistable of the
same phase. Here $D_C=L(-C)$ in the unshifted chamber and
$D_C=L(C)|_C$ in the shifted chamber.
\end{corollary}
A particularly simple consequence occurs on surfaces without negative
integral curves.
\begin{corollary}
\label{cor:no-negative-divisors}
If $X$ contains no integral curve of negative self-intersection, then
every line bundle is $\sigma_{B,\omega}$-stable for every divisorial stability condition $\sigma_{B,\omega}$.
\end{corollary}

The final consequence connects \Cref{thm:main} with the deformed
Hermitian--Yang--Mills equation and the large-scaling comparison studied by
Stoppa and Fan \cite{Stoppa,Fan}. For a line bundle $L$, the $B$-twisted
dHYM equation asks for a smooth representative of $c_1(L)-B$ whose
Lagrangian phase with respect to $\omega$ is constant. See
\cite{JacobYau,CJY} and, for the $B$-twisted convention,
\cite[Definition~2.1]{Fan}.

At a fixed scale, dHYM solvability and Bridgeland stability need not
coincide. On $\operatorname{Bl}_p\mathbb P^2$, Collins and Shi prove that
the existence of a dHYM solution implies Bridgeland stability, while the
converse need not hold
\cite[Theorem~4.15 and the example following it]{CollinsShi}.
A further distinction lies in scaling behavior. The dHYM equation is
invariant under
\[
  (\omega,c_1(L)-B)\longmapsto
  (k\omega,k(c_1(L)-B)),
\]
whereas Bridgeland stability of line bundles is not in general preserved
under the corresponding scaling. We therefore consider
\[
  (B,\omega,L)\longmapsto(kB,k\omega,L^{\otimes k}),
  \qquad k\in\mathbb Z_{>0}.
\]
We call the regime $k\to\infty$ the \emph{large-scaling limit}. For $B=0$,
this was introduced by Stoppa in the toric setting
\cite[Sections~1.2 and~6]{Stoppa}.

In the toric mirror-symmetry setting, the same scaling also suppresses
the quantum and $\widehat{\Gamma}$-class corrections entering the mirror
identification of central charges, leaving at leading order the period
central charge relevant to special Lagrangians
\cite[Remarks~2.1, 2.6(i), and~5.3]{Stoppa}. For background on the
mirror-symmetric origins of the dHYM equation and its supercritical
existence theory, see \cite{LYZ,Chen}.

Set $\alpha:=c_1(L)-B$ and, for an integral curve $C\subset X$, define
\begin{equation}\label{eq:Psi-dHYM-intro}
  \Psi_C(\alpha,\omega)
  :=(C\cdot\omega)(\omega^2-\alpha^2)
    +2(C\cdot\alpha)(\alpha\cdot\omega).
\end{equation}
By \cite[Proposition~8.5]{CJY}, in the $B$-twisted convention of
\cite[Definition~2.1, Theorem~2.2, and Remark~2.3(b)]{Fan}, the
dHYM equation is solvable exactly when
\[
  \Psi_C(\alpha,\omega)>0
\]
for every integral curve $C$. We call $L$
\emph{numerically $B$-twisted dHYM-semistable} if the corresponding weak
inequalities
\[
  \Psi_C(\alpha,\omega)\geq0
\]
hold for every integral curve $C$. Since $\Psi_C(\alpha,\omega)$ is linear
in $C$, the same inequality holds for every effective divisor.

Related effective-curve criteria are developed by Khalid and Sj\"ostr\"om
Dyrefelt \cite{KhalidSjoestroemDyrefelt}. They characterize solvability of
dHYM and more general $Z$-critical equations on compact K\"ahler surfaces by
finitely many such conditions and study the associated destabilizing curves
and wall-chamber structures.

\begin{corollary}
\label{cor:all-scalings}
With the notation above, the following are equivalent.
\begin{enumerate}[label=\textup{(\roman*)}]
  \item The line bundle $L$ is numerically $B$-twisted dHYM-semistable with
  respect to $\omega$.
  \item For every integer $k\geq1$, the object $L^{\otimes k}$ is
  $\sigma_{kB,k\omega}$-stable.
  \item The line bundle $L$ is stable in the large-scaling limit with
  respect to $(B,\omega)$, meaning that $L^{\otimes k}$ is
  $\sigma_{kB,k\omega}$-stable for all sufficiently large integers $k$.
\end{enumerate}
\end{corollary}
In particular, the implication from \textup{(i)} to \textup{(ii)} recovers
\cite[Theorem~2.9(b)]{Fan} without assuming Conjecture~2.8. Specializing to
$B=0$ and $\omega=[\omega_0]\in\NS(X)_\R$, our result removes the
Arcara--Miles assumption from the implication
\[
  \text{dHYM solvability on }L^\vee
  \Longrightarrow
  L^{\otimes k}[1]\text{ is }\sigma_{k[\omega_0]}\text{-stable}
\]
in \cite[Proposition~6.2]{Stoppa}. The other implication and the remaining
hypotheses of that proposition are unchanged.

We briefly describe the proof of \Cref{thm:main}. Tensoring by $L^{-1}$
reduces to $L=\cO_X$. In the unshifted chamber, choose a weak or strict
destabilizing subobject $E\hookrightarrow\cO_X$ of minimal positive rank.
Coherent cohomology gives an exact sequence
\[
  0\longrightarrow F\longrightarrow E
  \longrightarrow I_Z(-C)\longrightarrow0.
\]
The slope Harder--Narasimhan filtrations of $E$ and $F$, together with the
Bogomolov--Gieseker inequality, determine ordered Lorentzian vectors whose
cumulative defects record the values of a phase functional $\Lambda$ on
the filtration steps. The partial-sum estimate of
\Cref{thm:ordered-lorentzian-partial-sum} turns the sign conditions imposed
by rank minimality into timelike bounds for the corresponding total vectors.
If $\rk(E)\geq2$, the reverse triangle inequality then produces a strictly
destabilizing rank-one divisorial subobject, contradicting minimality.
Thus $\rk(E)=1$, and \Cref{lem:negative-square} shows that the resulting
non-zero effective divisor $C$ satisfies $C^2<0$. The shifted chamber
follows by derived duality.

The paper is organized as follows. In \Cref{sec:preliminaries} we recall the
tilted heart, characterize the natural divisorial subobjects, and analyze
subobjects of $\cO_X$. The phase functional and the Lorentzian lifts are constructed in \Cref{sec:phase-functional}. The ordered
Lorentzian partial-sum estimate is proved in \Cref{sec:lorentzian}. The
unshifted and shifted chambers are treated in
\Cref{sec:unshifted,sec:shifted}. The numerical criteria and the corollaries
above are proved in \Cref{sec:consequences}.

\subsection*{Acknowledgments}

I am grateful to my advisor, Lars Martin Sektnan, for his guidance and helpful comments on earlier drafts of this paper. I also thank Jacopo Stoppa for helpful comments and suggestions, and Rosa Sena-Dias for a conversation that led me to the problem considered here.

\section{Divisorial subobjects and reduction to the structure sheaf}
\label{sec:preliminaries}

\subsection{Tilted hearts and exactness}

For a torsion-free coherent sheaf $E$ of positive rank, put
\[
  \mu_\omega(E):=\frac{\omega\cdot c_1(E)}{\rk(E)}.
\]
Let $\mu_\omega^+(E)$ and $\mu_\omega^-(E)$ denote the largest and smallest
slopes of the factors in its $\mu_\omega$-Harder--Narasimhan filtration.
On a surface, intersection with the real ample class $\omega$ defines a
non-zero movable numerical curve class in the sense of
\cite[Definition~2.2]{GKP}: movability means non-negativity on effective
Cartier divisors. Strict positivity on non-zero effective divisors follows
separately from the ampleness of $\omega$. Existence and uniqueness
of the Harder--Narasimhan filtration therefore follow from
\cite[Corollary~2.27]{GKP}. We always take the saturated filtration supplied
there. In particular, its factors and every tail quotient are torsion-free.
For an arbitrary coherent sheaf $E$, write $E_{\mathrm{tf}}$ for its maximal
torsion-free quotient.

The torsion pair $(\cT_{B,\omega},\cF_{B,\omega})$ in $\Coh(X)$ is defined by
\begin{align*}
  \cT_{B,\omega}
  &:=\left\{E\in\Coh(X):
  \begin{array}{l}
    E_{\mathrm{tf}}=0,\text{ or}\\[-2pt]
    \mu_\omega^-(E_{\mathrm{tf}})>B\cdot\omega
  \end{array}
  \right\},\\
  \cF_{B,\omega}
  &:=\left\{E\in\Coh(X):
  \begin{array}{l}
    E=0,\text{ or}\\[-2pt]
    E\text{ is torsion-free of positive rank, and}\\[-2pt]
    \mu_\omega^+(E)\leq B\cdot\omega
  \end{array}
  \right\}.
\end{align*}
The tilted heart is
\begin{equation}\label{eq:tilted-heart}
  \cA_{B,\omega}
  :=\langle \cF_{B,\omega}[1],\cT_{B,\omega}\rangle.
\end{equation}
For a line bundle $M$,
\begin{equation}\label{eq:line-bundle-heart-membership}
  \begin{aligned}
    M\in\cA_{B,\omega}
    &\iff \omega\cdot(c_1(M)-B)>0,\\
    M[1]\in\cA_{B,\omega}
    &\iff \omega\cdot(c_1(M)-B)\leq0.
  \end{aligned}
\end{equation}

If $\ch(E)=(r,c,d)$, where $r=\ch_0(E)$, $c=\ch_1(E)$, and
$d=\ch_2(E)$, then expanding \eqref{eq:central-charge-integral} gives
\begin{equation}\label{eq:central-charge-expanded}
  Z_{B,\omega}(E)
  =\left(-d+B\cdot c+\frac r2(\omega^2-B^2)\right)
   +i\left(\omega\cdot c-rB\cdot\omega\right).
\end{equation}
Because $Z_{B,\omega}$ is a stability function on $\cA_{B,\omega}$ (see \cite[Section~2]{AB} and \cite[Section~6]{MacriSchmidt}), every non-zero $E\in\cA_{B,\omega}$
satisfies
\begin{equation}\label{eq:stability-function-range}
  Z_{B,\omega}(E)
  \in\{z\in\C:\im z>0\}\cup\R_{<0}.
\end{equation}

We shall repeatedly use the following consequence of the definition of a tilted heart.

\begin{lemma}
\label{lem:tilt-exactness}
Let $f:A\to A'$ be a morphism of coherent sheaves with $A,A'\in\cT_{B,\omega}$. If
\[
  \ker_{\Coh}(f)\in\cF_{B,\omega}
  \quad\text{and}\quad
  \operatorname{coker}_{\Coh}(f)\in\cT_{B,\omega},
\]
then $f$ is a monomorphism in $\cA_{B,\omega}$, and its quotient in the
tilted heart is the cone of $f$, whose only non-zero coherent cohomology
sheaves are
\[
  H^{-1}(\operatorname{cone}(f))=\ker_{\Coh}(f),
  \qquad
  H^0(\operatorname{cone}(f))=\operatorname{coker}_{\Coh}(f).
\]
In particular, a short exact sequence in $\Coh(X)$ all of whose terms belong
to $\cT_{B,\omega}$ remains short exact in $\cA_{B,\omega}$.
\end{lemma}
\begin{proof}
Set $Q:=\operatorname{cone}(f)$. Since $A$ and $A'$ are sheaves placed in
degree zero, the long exact sequence of coherent cohomology for the
triangle
\[
  A\xrightarrow{f}A'\longrightarrow Q\longrightarrow A[1]
\]
reduces to
\[
  0\longrightarrow H^{-1}(Q)\longrightarrow A\xrightarrow{f}A'
  \longrightarrow H^0(Q)\longrightarrow0.
\]
Thus $H^{-1}(Q)=\ker_{\Coh}(f)$,
$H^0(Q)=\operatorname{coker}_{\Coh}(f)$, and $H^i(Q)=0$ for
$i\neq-1,0$.

Recall that the tilted heart \eqref{eq:tilted-heart} has the concrete
description \cite[Lemma~6.3]{MacriSchmidt}
\[
  \cA_{B,\omega}
  =\left\{G\in\Db(X):
  \begin{array}{l}
    H^i(G)=0\text{ for }i\neq-1,0,\\
    H^{-1}(G)\in\cF_{B,\omega},\quad
    H^0(G)\in\cT_{B,\omega}
  \end{array}
  \right\}.
\]
Therefore we have that $Q\in\cA_{B,\omega}$. Moreover,
$A,A'\in\cT_{B,\omega}$ imply that $A$ and $A'$, viewed as complexes in
degree zero, also lie in $\cA_{B,\omega}$. Hence all three terms of the
displayed triangle lie in the heart. Applying the cohomology functors of the
tilted $t$-structure, or equivalently using the standard correspondence
between triangles in a heart and extensions, gives a short exact sequence
\[
  0\longrightarrow A\xrightarrow{f}A'\longrightarrow Q\longrightarrow0
\]
in $\cA_{B,\omega}$. In particular, $f$ is a monomorphism and its quotient
is $Q$.

Finally, if $0\to A\to A'\to A''\to0$ is a short exact sequence in
$\Coh(X)$ with $A,A',A''\in\cT_{B,\omega}$, then the coherent kernel of
$A\to A'$ is zero, which belongs to $\cF_{B,\omega}$, and its coherent
cokernel is $A''\in\cT_{B,\omega}$. The preceding argument gives the claimed
short exact sequence in the tilted heart.
\end{proof}

\subsection{Divisorial subobjects in the tilted heart}

The statement of \Cref{thm:main} picks out two natural morphisms associated
with an effective divisor $C$,
\[
L(-C)\longrightarrow L
\qquad\text{and}\qquad
L(C)|_C\longrightarrow L[1].
\]
For these morphisms to serve as potential destabilizing subobjects, they
must first be monomorphisms in the tilted heart. The following lemma
expresses this categorical condition in terms of the numerical position of
$L$ relative to the slope defining the tilt.

\begin{lemma}
\label{lem:divisorial-subobject-criterion}
Let $L$ be a line bundle, let $C$ be a non-zero effective Cartier divisor,
and put $s=\omega\cdot(c_1(L)-B)$.
\begin{enumerate}[label=\textup{(\alph*)}]
  \item Assume $s>0$. The natural morphism
  \[
    L(-C)\longrightarrow L
  \]
  is a monomorphism in $\cA_{B,\omega}$ if and only if
  \[
    s-\omega\cdot C>0.
  \]
  When this holds, its quotient is $L|_C$.

  \item Assume $s\leq0$. The natural morphism
  \[
    L(C)|_C\longrightarrow L[1]
  \]
  is a monomorphism in $\cA_{B,\omega}$ if and only if
  \[
    s+\omega\cdot C\leq0.
  \]
  When this holds, its quotient is $L(C)[1]$.
\end{enumerate}
\end{lemma}

\begin{proof}
For (a), the assumption $s>0$ says that $L\in\cT_{B,\omega}$,
so $L$ is an unshifted object of $\cA_{B,\omega}$. Consider the standard
coherent exact sequence
\[
  0\longrightarrow L(-C)\longrightarrow L\longrightarrow L|_C
  \longrightarrow0
\]
associated with the effective Cartier divisor $C$. The quotient $L|_C$ is
a torsion sheaf and hence belongs to $\cT_{B,\omega}$. Moreover,
$L(-C)$ is a line bundle, so
\[
  L(-C)\in\cT_{B,\omega}
  \quad\Longleftrightarrow\quad
  \omega\cdot(c_1(L)-C-B)=s-\omega\cdot C>0.
\]
When this inequality holds, all three sheaves in the coherent sequence lie
in $\cT_{B,\omega}$, so \Cref{lem:tilt-exactness} identifies it with a
short exact sequence in $\cA_{B,\omega}$. Thus $L(-C)\to L$ is a
monomorphism there with quotient $L|_C$.

Conversely, if the natural morphism is a monomorphism in the tilted heart,
then its source $L(-C)$ must itself be an object of that heart. Since it is
a sheaf concentrated in degree zero, we have $L(-C)\in\cT_{B,\omega}$. Equivalently, by
\eqref{eq:line-bundle-heart-membership}, $s-\omega\cdot C>0$.

For (b), the assumption $s\leq0$ gives $L\in\cF_{B,\omega}$ and hence
$L[1]\in\cA_{B,\omega}$. Start with the coherent exact sequence
\[
  0\longrightarrow L\longrightarrow L(C)\longrightarrow L(C)|_C
  \longrightarrow0
\]
and its associated triangle
\[
  L\longrightarrow L(C)\longrightarrow L(C)|_C\longrightarrow L[1].
\]
The inequality $s+\omega\cdot C\leq0$ is equivalent to
$L(C)\in\cF_{B,\omega}$. Thus, under this inequality, both $L[1]$ and
$L(C)[1]$ belong to $\cA_{B,\omega}$, while the torsion sheaf $L(C)|_C$
belongs to $\cT_{B,\omega}\subset\cA_{B,\omega}$. Rotating the last
triangle therefore produces a triangle all of whose terms lie in the
tilted heart, and hence a short exact sequence
\[
  0\longrightarrow L(C)|_C\longrightarrow L[1]
  \longrightarrow L(C)[1]\longrightarrow0
\]
in $\cA_{B,\omega}$. This proves the forward implication and identifies
the quotient with $L(C)[1]$.

Conversely, the natural morphism $L(C)|_C\to L[1]$ is the boundary
morphism in the displayed coherent triangle, whose cone is $L(C)[1]$. If
it is a monomorphism in the tilted heart, this cone is its quotient and
must belong to $\cA_{B,\omega}$. By
\eqref{eq:line-bundle-heart-membership}, this is equivalent to
$s+\omega\cdot C\leq0$.
\end{proof}

\subsection{Reduction to the structure sheaf}

\begin{lemma}
\label{lem:tensor-reduction}
Let $M$ be a line bundle with $m=c_1(M)$. Tensoring by $M^{-1}$ identifies
the stability problem for $\sigma_{B,\omega}$ with that for
$\sigma_{B-m,\omega}$. More precisely,
\begin{align*}
  Z_{B-m,\omega}(E\otimes M^{-1})&=Z_{B,\omega}(E),\\
  E\in\cA_{B,\omega}&\iff E\otimes M^{-1}\in\cA_{B-m,\omega}.
\end{align*}
Under this identification,
\[
  M\mapsto\cO_X,
  \qquad M(-C)\mapsto\cO_X(-C),
  \qquad M(C)|_C\mapsto\cO_X(C)|_C.
\]
\end{lemma}

\begin{proof}
Since $\ch(E\otimes M^{-1})=e^{-m}\ch(E)$,
\[
  e^{-(B-m+i\omega)}\ch(E\otimes M^{-1})
  =e^{-(B+i\omega)}\ch(E),
\]
which proves the identity of central charges. Tensoring shifts every slope
by $-m\cdot\omega$, so it identifies the two slope torsion pairs and hence
the two tilted hearts. The final assertions are tautological.
\end{proof}

By \Cref{lem:tensor-reduction}, it is enough to prove \Cref{thm:main} for $L=\cO_X$.

\subsection{Structure of arbitrary subobjects of the structure sheaf}

The following standard description is implicit in the rank-reduction
arguments of \cite[Section~4]{AM}, and we include the details needed below.

\begin{lemma}
\label{lem:subobject-structure}
Assume $\cO_X\in\cA_{B,\omega}$, and let
\begin{equation}\label{eq:heart-subobject-sequence}
  0\longrightarrow E\longrightarrow\cO_X\longrightarrow Q\longrightarrow0
\end{equation}
be a short exact sequence in $\cA_{B,\omega}$ with $E\neq0$. Then:
\begin{enumerate}[label=\textup{(\roman*)}]
  \item $E$ is a torsion-free coherent sheaf in $\cT_{B,\omega}$ and has positive rank.
  \item If $F:=H^{-1}(Q)$, then $F$ is a torsion-free sheaf in
  $\cF_{B,\omega}$, while $H^0(Q)$ is a torsion sheaf in
  $\cT_{B,\omega}$.
  \item There exist an effective Cartier divisor $C$, possibly zero, and a
  zero-dimensional subscheme $Z\subset X$ such that
  \begin{equation}\label{eq:structure-sheaf-sequence}
    0\longrightarrow F\longrightarrow E
      \longrightarrow I_Z(-C)\longrightarrow0
  \end{equation}
  is exact in $\Coh(X)$.
  \item The sheaf $I_Z(-C)$ belongs to $\cT_{B,\omega}$, and
  \[
    \rk(F)=\rk(E)-1.
  \]
\end{enumerate}
\end{lemma}

\begin{proof}
Taking coherent cohomology of \eqref{eq:heart-subobject-sequence} gives
\begin{equation}\label{eq:cohomology-subobject}
  0\longrightarrow F:=H^{-1}(Q)\longrightarrow H^0(E)
  \longrightarrow\cO_X\longrightarrow H^0(Q)\longrightarrow0,
\end{equation}
while $H^{-1}(E)=0$. Thus $E=H^0(E)$ is a coherent sheaf in
$\cT_{B,\omega}$ and $F\in\cF_{B,\omega}$.

The image $I$ of $E\to\cO_X$ is a non-zero rank-one torsion-free subsheaf
of $\cO_X$. It is non-zero because a torsion sheaf admits no non-zero map
to the torsion-free sheaf $\cO_X$, whereas the zero morphism cannot be a
monomorphism in the heart unless $E=0$. Hence $E$ has positive rank. On a
smooth surface, every rank-one torsion-free subsheaf of $\cO_X$ is of the
form
\[
  I=I_Z(-C)
\]
for an effective Cartier divisor $C$, possibly zero, and a zero-dimensional
subscheme $Z$. Indeed, $I^{\vee\vee}$ is a line bundle contained in $\cO_X$,
hence equals $\cO_X(-C)$ for an effective Cartier divisor $C$ (possibly zero),
and
$I^{\vee\vee}/I$ is zero-dimensional. Since $I$ has rank one,
$H^0(Q)$ is a rank-zero quotient of $\cO_X$, hence a torsion sheaf. The
first three non-zero terms of \eqref{eq:cohomology-subobject} now yield
\eqref{eq:structure-sheaf-sequence}.

Both $F$ and $I_Z(-C)$ are torsion-free, so their extension $E$ is
torsion-free. Since $I_Z(-C)$ is a coherent quotient of $E$ and
$\cT_{B,\omega}$ is quotient-closed, it belongs to $\cT_{B,\omega}$. The
torsion sheaf $H^0(Q)$ belongs to $\cT_{B,\omega}$ by definition. Finally,
the rank identity follows from \eqref{eq:structure-sheaf-sequence}.
\end{proof}

\section{The phase functional and Lorentzian lifts}
\label{sec:phase-functional}

We now develop the numerical and Lorentzian tools needed to prove the
forward implications of \Cref{thm:main}(a). By the tensor reduction of
\Cref{lem:tensor-reduction}, it is enough to consider $L=\cO_X$. The
unshifted chamber is then characterized by
\[
  \cO_X\in\cA_{B,\omega},
  \qquad\text{equivalently}\qquad B\cdot\omega<0.
\]
We first encode phase comparison with the structure sheaf by a linear
functional. We then use the Bogomolov--Gieseker inequality to realize its
values on Harder--Narasimhan factors as Lorentzian square defects. These
constructions will be used in \Cref{sec:lorentzian,sec:unshifted}. The
shifted chamber will be treated separately in \Cref{sec:shifted}.
\subsection{Phase comparison and normalization}

Set the real numbers
\[
  h:=\omega^2>0,
  \qquad p:=-B\cdot\omega>0,
  \qquad q:=\frac{h-B^2}{2p},
\]
and define the numerical divisor classes
\[
  \Gamma:=-B+q\omega,
  \qquad e:=\frac{\omega}{\sqrt h}
\]
in $\NS(X)_\R$. Here products and squares of divisor classes are taken with
respect to the intersection pairing. Since $2pq=h-B^2$, a direct calculation
gives
\[
  \Gamma^2
  = (1+q^2)h>0,
  \qquad e^2=1.
\]
We may therefore set the real numbers
\[
  a:=\sqrt{\Gamma^2}=\sqrt{(1+q^2)h}>0,
  \qquad \tau:=q\sqrt h.
\]
Then $\tau^2=q^2h<a^2$, and hence $\abs{\tau}<a$. Moreover,
\begin{equation}\label{eq:Z-O}
  Z_{B,\omega}(\cO_X)=p(q+i).
\end{equation}

For a numerical Chern character $v=(r,c,d)$, define the linear functional
\begin{equation*}
  \Lambda(v):=d+\Gamma\cdot c.
\end{equation*}
We write $\Lambda(E):=\Lambda(\ch(E))$.

\begin{lemma}
Assume $B\cdot\omega<0$. For every non-zero $E\in\cA_{B,\omega}$,
\begin{align}
  \Arg Z_{B,\omega}(E)\geq\Arg Z_{B,\omega}(\cO_X)
  &\quad\Longleftrightarrow\quad \Lambda(E)\geq0,
  \label{eq:phase-lambda-equivalence}\\
  \Arg Z_{B,\omega}(E)>\Arg Z_{B,\omega}(\cO_X)
  &\quad\Longleftrightarrow\quad \Lambda(E)>0,
  \label{eq:strict-phase-lambda-equivalence}\\
  \Arg Z_{B,\omega}(E)=\Arg Z_{B,\omega}(\cO_X)
  &\quad\Longleftrightarrow\quad \Lambda(E)=0.\notag
\end{align}
\end{lemma}

\begin{proof}
Put
\[
  z:=Z_{B,\omega}(E),
  \qquad
  z_0:=Z_{B,\omega}(\cO_X).
\]
Both numbers are non-zero: this follows from
\eqref{eq:stability-function-range}, since $E$ and $\cO_X$ are non-zero
objects of $\cA_{B,\omega}$. Write $z=x+iy$. By
\eqref{eq:central-charge-expanded},
\begin{align*}
  x&=-d+B\cdot c+rpq,\\
  y&=\omega\cdot c+rp.
\end{align*}
Using \eqref{eq:Z-O},
\begin{align*}
  \im(z\ol{z_0})
  &=p(yq-x)\\
  &=p\bigl(d+(q\omega-B)\cdot c\bigr)\\
  &=p\Lambda(E).
\end{align*}
By \eqref{eq:stability-function-range}, there are uniquely determined
angles
\[
  \theta:=\Arg z\in(0,\pi],
  \qquad
  \theta_0:=\Arg z_0\in(0,\pi].
\]
Consequently $\theta-\theta_0\in(-\pi,\pi)$, and
\[
  \im(z\ol{z_0})
  =\abs z\,\abs{z_0}\sin(\theta-\theta_0).
\]
On the interval $(-\pi,\pi)$, $\sin t$ has the same sign as $t$ and
vanishes only at $t=0$. Hence
\[
  \operatorname{sgn}\Lambda(E)
  =\operatorname{sgn}(\theta-\theta_0),
\]
because $p>0$. This proves the weak, strict, and equality equivalences.
\end{proof}

For an effective divisor $C$,
\begin{equation}\label{eq:Lambda-line-divisor}
  \Lambda(\cO_X(-C))
  =\frac{C^2}{2}-\Gamma\cdot C
  =\frac{(\Gamma-C)^2-\Gamma^2}{2}.
\end{equation}
If $Z$ is zero-dimensional of length $\ell(Z)$, then
\begin{equation}\label{eq:Lambda-ideal-divisor}
  \Lambda(I_Z(-C))
  =\Lambda(\cO_X(-C))-\ell(Z).
\end{equation}

The following example shows that the term involving $\ell(Z)$ in
\eqref{eq:Lambda-ideal-divisor} can genuinely occur, and also illustrates why
\Cref{thm:main} gives a detection result rather than a classification of all
destabilizing subobjects.

\begin{example}
\label{ex:non-line-bundle-destabilizer}
Let $X=\operatorname{Bl}_{p_0}\mathbb P^2$, let $H$ be the pullback of a line,
let $E$ be the exceptional curve, and choose a closed point
$x_0\notin E$. Set
\[
  \omega=H-\frac12E,
  \qquad
  B=-7H.
\]
Then $\omega$ is ample and $B\cdot\omega=-7$. The sheaves
$\cO_X$, $\cO_X(-E)$, and $I_{x_0}(-E)$ all belong to
$\cT_{B,\omega}$, since their slopes are respectively $0$, $-1/2$, and
$-1/2$, all strictly larger than $-7$. The coherent inclusions
\[
  I_{x_0}(-E)\subset\cO_X(-E)\subset\cO_X
\]
have torsion cokernels. By \Cref{lem:tilt-exactness}, both
$I_{x_0}(-E)$ and $\cO_X(-E)$ are therefore subobjects of $\cO_X$ in
$\cA_{B,\omega}$.

In the notation above,
\[
  h=\frac34,
  \qquad p=7,
  \qquad q=-\frac{193}{56},
  \qquad
  \Gamma=\frac{199}{56}H+\frac{193}{112}E.
\]
Using $H^2=1$, $H\cdot E=0$, $E^2=-1$, and the fact that $x_0$ has
length one, equations \eqref{eq:Lambda-line-divisor} and
\eqref{eq:Lambda-ideal-divisor} give
\[
  \Lambda(\cO_X(-E))=\frac{137}{112},
  \qquad
  \Lambda(I_{x_0}(-E))=\frac{25}{112}>0.
\]
Hence \eqref{eq:strict-phase-lambda-equivalence} shows that
$I_{x_0}(-E)$ strictly destabilizes $\cO_X$. Moreover,
\[
  Z_{B,\omega}(\cO_X(-E))=-\frac{189}{8}+\frac{13}{2}i,
  \qquad
  Z_{B,\omega}(I_{x_0}(-E))=-\frac{181}{8}+\frac{13}{2}i.
\]
The two central charges have the same positive imaginary part, and the
first has smaller real part. Therefore,
\[
  \Arg Z_{B,\omega}(\cO_X(-E))
  >\Arg Z_{B,\omega}(I_{x_0}(-E))
  >\Arg Z_{B,\omega}(\cO_X).
\]
Thus $I_{x_0}(-E)$ is a strict destabilizing subobject that is not locally
free, while its reflexive hull
\[
  \bigl(I_{x_0}(-E)\bigr)^{\vee\vee}=\cO_X(-E)
\]
is a divisorial subobject of still larger phase.
\end{example}

\subsection{Lorentzian lifts of slope-semistable sheaves}
\label{sec:bogomolov}

Let $G$ be a torsion-free $\mu_\omega$-semistable sheaf of rank $n>0$, and write
\[
  \ch(G)=(n,c,d).
\]
Its discriminant is
\begin{equation*}
  \Delta(G):=c^2-2nd.
\end{equation*}
Viewing $\omega$ as the corresponding non-zero movable curve class, the
Bogomolov--Gieseker inequality \cite[Theorem~5.1]{GKP} gives
\begin{equation*}
  \Delta(G)\geq0.
\end{equation*}
This agrees with the discriminant convention in \cite{GKP}, since
\[
  c_1(G)^2-2n\ch_2(G)
  =2n c_2(G)-(n-1)c_1(G)^2.
\]

Define
\begin{equation*}
  U_G:=\Gamma+\frac cn\in\NS(X)_\R.
\end{equation*}
Then
\begin{equation}\label{eq:UG-Lambda}
  \frac{2\Lambda(G)}n
  =U_G^2-\Gamma^2-\frac{\Delta(G)}{n^2}.
\end{equation}

Let
\[
  \wt V:=\NS(X)_\R\oplus\R\eps,
  \qquad \eps^2=-1,
  \qquad \eps\perp\NS(X)_\R.
\]
By the Hodge index theorem \cite[Theorem~V.1.9]{Hartshorne}, $\NS(X)_\R$ has signature $(1,\rho(X)-1)$, so
$\wt V$ has signature $(1,\rho(X))$. Extend $e$ by declaring
$e\cdot\eps=0$, and define the Lorentzian lift of $G$ to be
\begin{equation}\label{eq:augmented-vector}
  \wt U_G:=U_G+\frac{\sqrt{\Delta(G)}}n\eps.
\end{equation}
The Bogomolov inequality ensures that this lift is real. Equations
\eqref{eq:UG-Lambda} and \eqref{eq:augmented-vector} give
\begin{equation}\label{eq:augmented-square}
  \wt U_G^2=a^2+\frac{2\Lambda(G)}n.
\end{equation}
For $\wt x\in\wt V$, we call $e\cdot\wt x$ its time coordinate. The time
coordinate of $\wt U_G$ is
\begin{align*}
  e\cdot\wt U_G
  &=\frac{\Gamma\cdot\omega+\mu_\omega(G)}{\sqrt h}\\
  &=\frac{p+qh+\mu_\omega(G)}{\sqrt h}.
\end{align*}
Because $B\cdot\omega=-p$, we obtain the comparison
\begin{equation}\label{eq:time-threshold}
  \begin{aligned}
  \mu_\omega(G)>B\cdot\omega&\quad\Longrightarrow\quad
  e\cdot\wt U_G>\tau,\\
  \mu_\omega(G)\leq B\cdot\omega&\quad\Longrightarrow\quad
  e\cdot\wt U_G\leq\tau.
  \end{aligned}
\end{equation}
The displayed formula also shows that Harder--Narasimhan slope ordering gives
decreasing time coordinates. Since $\abs{\tau}<a$,
\eqref{eq:time-threshold} yields
$e\cdot\wt U_G>-a$ when $\mu_\omega(G)>B\cdot\omega$ and
$e\cdot\wt U_G<a$ when $\mu_\omega(G)\leq B\cdot\omega$. These are the
ordering and one-sided bounds used in the next section.

We use the elementary observation that projection from $\wt V$ to
$\NS(X)_\R$ can only increase the square. Namely, if $\wt x=x+s\eps$,
then
\begin{equation}\label{eq:projection-increases-square}
  x^2=\wt x^2+s^2\geq\wt x^2,
\end{equation}
and the time coordinate is unchanged.

\section{An ordered Lorentzian partial-sum estimate}
\label{sec:lorentzian}

For background on Lorentzian geometry, see, for example, \cite[Chapter~5]{ONeill}. We use the sign convention in which timelike
vectors have positive square.

Let $V$ be a real vector space with a non-degenerate symmetric form of
signature $(1,N)$, and fix $e\in V$ with $e^2=1$. Every $x\in V$ has a
unique decomposition
\[
  x=t e+\xi,
  \qquad t=e\cdot x,
  \qquad \xi\in e^\perp.
\]
We write
\[
  \norm{\xi}:=\sqrt{-\xi^2}.
\]
We call $x$ \emph{causal} if $x^2\geq0$ and \emph{timelike} if
$x^2>0$. A non-zero causal vector is \emph{future-pointing}, respectively
\emph{past-pointing}, when its time coordinate $e\cdot x$ is positive,
respectively negative.

Fix $a>0$, let $m\geq1$, let $x_1,\dots,x_m\in V$, let
$n_1,\dots,n_m>0$, and put
\[
  X:=\sum_{i=1}^m n_i x_i.
\]
If the vectors $x_i$ all belonged to the same time cone and satisfied the
termwise bounds $x_i^2\geq a^2$, the reverse triangle inequality would
immediately give
\[
  X^2\geq a^2\left(\sum_{i=1}^m n_i\right)^2.
\]

The Harder--Narasimhan factors arising in our application provide weaker
information: their time coordinates are ordered, but the square defects
$x_i^2-a^2$ are controlled only through certain partial sums. In
particular, some individual defects may be negative. The estimate proved
below shows that the appropriate partial-sum conditions, together with the
ordering and a one-sided bound on the time coordinates, still force the same
quadratic lower bound for $X$ and determine whether $X$ is future- or past-pointing.

\subsection{Ordered convexity and the lower null coordinate}
\label{subsec:ordered-convexity}

To turn these partial-sum conditions into the required lower bound for the weighted sum of the lower null coordinates, we use convexity and Abel summation. We first record the elementary convexity statement behind this argument and then apply it using the ordering of the supporting slopes induced by the time coordinates.

\begin{lemma}
\label{lem:ordered-convex-comparison}
Let $m\geq1$, let $n_1,\dots,n_m>0$, and let
$b,s_1,\dots,s_m\in\R$. For each $i$, let $I_i\subseteq\R$ be an interval
containing $b$ and $s_i$, and let $f_i:I_i\to\R$ be convex and
differentiable at $b$. Put
\[
  d_i:=n_i(s_i-b),
  \qquad
  D_j:=\sum_{i=1}^j d_i,
  \qquad
  q_i:=f_i'(b).
\]
Then
\begin{equation}\label{eq:ordered-supporting-line-bound}
  \sum_{i=1}^m n_i\bigl(f_i(s_i)-f_i(b)\bigr)
  \geq
  D_mq_m+\sum_{j=1}^{m-1}D_j(q_j-q_{j+1}).
\end{equation}
Consequently, suppose that $q_m\geq0$ and that one of the following sets of
hypotheses holds:
\begin{enumerate}[label=\textup{(\alph*)}]
  \item $q_1\leq\cdots\leq q_m$, $D_j\leq0$ for $j<m$, and $D_m\geq0$.
  \item $q_1\geq\cdots\geq q_m$ and $D_j\geq0$ for every $j$.
\end{enumerate}
Then
\begin{equation*}
  \sum_{i=1}^m n_i f_i(s_i)
  \geq
  \sum_{i=1}^m n_i f_i(b).
\end{equation*}
\end{lemma}

\begin{proof}
Since $f_i$ is convex and differentiable at $b$, its tangent line at $b$
is a supporting line. Hence, for each $i$,
\[
  f_i(s_i)\geq f_i(b)+q_i(s_i-b).
\]
Multiplying by $n_i$ and summing over $i$ therefore gives
\[
  \sum_{i=1}^m n_i\bigl(f_i(s_i)-f_i(b)\bigr)
  \geq \sum_{i=1}^m n_iq_i(s_i-b)
  =\sum_{i=1}^m d_iq_i.
\]
Set $D_0:=0$. Since $d_i=D_i-D_{i-1}$, Abel summation may be written
explicitly as
\begin{equation*}
  \begin{aligned}
    \sum_{i=1}^m d_iq_i
    &=\sum_{i=1}^m(D_i-D_{i-1})q_i\\
    &=D_mq_m+\sum_{j=1}^{m-1}D_j(q_j-q_{j+1}).
  \end{aligned}
\end{equation*}
Combining the last two displays proves
\eqref{eq:ordered-supporting-line-bound}.

It remains to check the signs in the two cases. In case~\textup{(a)}, the
sequence $(q_j)$ is non-decreasing, so $q_j-q_{j+1}\leq0$. Since
$D_j\leq0$ for $j<m$, every product
$D_j(q_j-q_{j+1})$ is non-negative. The remaining term $D_mq_m$ is also
non-negative because $D_m\geq0$ and $q_m\geq0$. In case~\textup{(b)}, the
sequence $(q_j)$ is non-increasing, so $q_j-q_{j+1}\geq0$, while every
$D_j$ is non-negative. Again all terms on the right-hand side of
\eqref{eq:ordered-supporting-line-bound}, including $D_mq_m$, are
non-negative. Thus in either case the left-hand side of
\eqref{eq:ordered-supporting-line-bound} is non-negative, which is exactly
the asserted inequality.
\end{proof}

For $t\geq a$ and $s\leq t^2$, set
\begin{equation*}
  \ell(t,s):=t-\sqrt{t^2-s}.
\end{equation*}
For a vector $x=te+\xi$ of square $s$, this is the lower null coordinate of
the radial vector $(t,\norm{\xi})$. For $t>a$ and $s<t^2$,
\begin{equation*}
  \frac{\partial\ell}{\partial s}(t,s)
  =\frac{1}{2\sqrt{t^2-s}},
  \qquad
  \frac{\partial^2\ell}{\partial s^2}(t,s)
  =\frac{1}{4(t^2-s)^{3/2}}>0.
\end{equation*}
In particular, $s\mapsto\ell(t,s)$ is increasing and convex. At the base
point $s=a^2$, its supporting slope is
\begin{equation}\label{eq:null-baseline-slope}
  q_a(t):=\frac{1}{2\sqrt{t^2-a^2}},
\end{equation}
which is decreasing in $t\in(a,\infty)$. In the future case, the time coordinates form a non-increasing sequence, so the corresponding supporting slopes form a non-decreasing sequence. Together with the signs of the partial sums of the defects, this is exactly \Cref{lem:ordered-convex-comparison}(a). In the past case, the reversed ordering and partial-sum signs give \Cref{lem:ordered-convex-comparison}(b).

\begin{lemma}
\label{lem:null-coordinate-comparison}
Fix $a>0$, let $m\geq1$, $n_1,\dots,n_m>0$, and let
$s_1,\dots,s_m\in\R$. Put
\[
  D_j:=\sum_{i=1}^j n_i(s_i-a^2).
\]
\begin{enumerate}[label=\textup{(\alph*)}]
  \item Suppose
  \[
    t_1\geq\cdots\geq t_m\geq a,
    \qquad s_i\leq t_i^2,
    \qquad D_j\leq0\ (j<m),
    \qquad D_m\geq0.
  \]
  Then
  \begin{equation}\label{eq:null-comparison-future}
    \sum_{i=1}^m n_i\ell(t_i,s_i)
    \geq
    \sum_{i=1}^m n_i\ell(t_i,a^2).
  \end{equation}

  \item Suppose
  \[
    a\leq u_1\leq\cdots\leq u_m,
    \qquad s_i\leq u_i^2,
    \qquad D_j\geq0\quad(1\leq j\leq m).
  \]
  Then
  \begin{equation}\label{eq:null-comparison-past}
    \sum_{i=1}^m n_i\ell(u_i,s_i)
    \geq
    \sum_{i=1}^m n_i\ell(u_i,a^2).
  \end{equation}
\end{enumerate}
\end{lemma}

\begin{proof}
Fix $\varepsilon>0$. In part (a), apply
\Cref{lem:ordered-convex-comparison}(a) with $b=a^2$ and
\[
  f_i(s):=\ell(t_i+\varepsilon,s).
\]
These functions are convex on intervals containing both $a^2$ and $s_i$.
Their supporting slopes at $a^2$ are positive and, by
\eqref{eq:null-baseline-slope}, form a non-decreasing sequence because the
$t_i$ form a non-increasing sequence. Hence
\[
  \sum_i n_i\ell(t_i+\varepsilon,s_i)
  \geq
  \sum_i n_i\ell(t_i+\varepsilon,a^2).
\]
Letting $\varepsilon\to 0$ proves
\eqref{eq:null-comparison-future}. The regularization also covers the case
$t_i=a$, where the graph has a vertical tangent at $s=a^2$.

For part (b), use instead $f_i(s)=\ell(u_i+\varepsilon,s)$. Since the $u_i$
are non-decreasing, their supporting slopes form a non-increasing positive
sequence. Part (b) of \Cref{lem:ordered-convex-comparison}, followed by
$\varepsilon\to 0$, gives \eqref{eq:null-comparison-past}.
\end{proof}

\subsection{The ordered Lorentzian partial-sum estimate}
\label{subsec:ordered-partial-sum}

The two sign patterns in the theorem below are exactly those that will
arise from the Harder--Narasimhan filtrations in
\Cref{sec:unshifted}. Indeed, if
\[
  0=E_0\subset E_1\subset\cdots\subset E_m=E
\]
is the filtration of a minimal-rank destabilizer, with factors
$G_i=E_i/E_{i-1}$ and ranks $n_i$, then
\eqref{eq:augmented-square} gives
\[
  \sum_{i=1}^j n_i\bigl(\wt U_{G_i}^2-a^2\bigr)
  =2\Lambda(E_j).
\]
Rank minimality will give non-positive proper partial sums and a
non-negative total sum. For the Harder--Narasimhan filtration of the kernel
sheaf $F$ appearing in \eqref{eq:structure-sheaf-sequence}, it will instead
give non-negative partial sums at every stage.
The slope ordering and \eqref{eq:time-threshold} yield the corresponding
ordering and one-sided bounds for the time coordinates.

The preceding null-coordinate comparison converts these partial-sum
conditions into a radial estimate. Combining this with the ordinary triangle
inequality in $e^\perp$ gives the Lorentzian bound below.

\begin{theorem}[Ordered Lorentzian partial-sum estimate]
\label{thm:ordered-lorentzian-partial-sum}
Let $V$ have signature $(1,N)$, let $e^2=1$, and fix $a>0$. Let
$m\geq1$, let $x_1,\dots,x_m\in V$, let $n_1,\dots,n_m>0$, and put
\[
  \begin{aligned}
    t_i&:=e\cdot x_i,
    & s_i&:=x_i^2,
    & D_0&:=0,\\
    D_j&:=\sum_{i=1}^j n_i(s_i-a^2),
    & X&:=\sum_{i=1}^m n_i x_i.
  \end{aligned}
\]

\begin{enumerate}[label=\textup{(\roman*)}]
  \item \emph{Future case.}
  Assume
  \[
    t_1\geq\cdots\geq t_m>-a,
    \qquad D_j\leq0\ (j<m),
    \qquad D_m\geq0.
  \]
  Then necessarily $t_i\geq a$ for every $i$. Moreover, $X$ is
  future-pointing and timelike, and
  \begin{align}
    X^2
    &\geq
    \left(\sum_i n_i\bigl(t_i-\sqrt{t_i^2-a^2}\bigr)\right)
    \left(\sum_i n_i\bigl(t_i+\sqrt{t_i^2-a^2}\bigr)\right)
    \label{eq:future-refined-partial-sum-bound}\\
    &\geq a^2\left(\sum_i n_i\right)^2.
    \label{eq:future-partial-sum-bound}
  \end{align}

  \item \emph{Past case.}
  Assume
  \[
    t_1\geq\cdots\geq t_m,
    \qquad t_1<a,
    \qquad D_j\geq0\quad(1\leq j\leq m).
  \]
  Then necessarily $t_i\leq-a$ for every $i$. Moreover, writing
  $u_i:=-t_i$, the vector $-X$ is future-pointing and timelike, and
  \begin{align}
    X^2
    &\geq
    \left(\sum_i n_i\bigl(u_i-\sqrt{u_i^2-a^2}\bigr)\right)
    \left(\sum_i n_i\bigl(u_i+\sqrt{u_i^2-a^2}\bigr)\right)
    \label{eq:past-refined-partial-sum-bound}\\
    &\geq a^2\left(\sum_i n_i\right)^2.
    \label{eq:past-partial-sum-bound}
  \end{align}
\end{enumerate}
\end{theorem}

\begin{proof}
For the future version, the last defect is non-negative because
\[
  n_m(s_m-a^2)=D_m-D_{m-1}\geq0.
\]
Thus $s_m\geq a^2$. Since $s_m=t_m^2-\norm{\xi_m}^2$ and $t_m>-a$, this
forces $t_m\geq a$, and hence $t_i\geq a$ for all $i$.

Write $x_i=t_i e+\xi_i$ and set
\[
  \begin{aligned}
    r_i&:=\norm{\xi_i}=\sqrt{t_i^2-s_i},
    & T&:=\sum_i n_i t_i,\\
    R&:=\sum_i n_i r_i,
    & U&:=T-R.
  \end{aligned}
\]
By \Cref{lem:null-coordinate-comparison}(a),
\begin{equation*}
  U=\sum_i n_i\ell(t_i,s_i)
  \geq
  U_0:=\sum_i n_i\ell(t_i,a^2)>0.
\end{equation*}
Also $U\leq T$. The ordinary triangle inequality in the negative-definite
space $e^\perp$ gives
\[
  \norm{\sum_i n_i\xi_i}\leq R.
\]
Therefore
\begin{align*}
  X^2
  &=T^2-\norm{\sum_i n_i\xi_i}^2\\
  &\geq T^2-R^2
   =U(2T-U)\\
  &\geq U_0(2T-U_0),
\end{align*}
where the last step uses that $v\mapsto v(2T-v)$ is increasing on
$[0,T]$. Since
\[
  U_0=\sum_i n_i\bigl(t_i-\sqrt{t_i^2-a^2}\bigr)
\]
and
\[
  2T-U_0
  =\sum_i n_i\bigl(t_i+\sqrt{t_i^2-a^2}\bigr),
\]
this proves \eqref{eq:future-refined-partial-sum-bound}. Each pair of factors in
the two sums has product $a^2$, so Cauchy--Schwarz yields
\[
  U_0(2T-U_0)
  \geq
  \left(\sum_i n_i a\right)^2,
\]
which is \eqref{eq:future-partial-sum-bound}. Finally, $T>0$, so $X$ is
future-pointing.

For the past version, the first partial-sum condition gives $s_1\geq a^2$.
Since $t_1<a$, it follows that $t_1\leq-a$, and the ordering gives
$t_i\leq-a$ for all $i$. Put $y_i:=-x_i$ and $u_i=e\cdot y_i=-t_i$, then
\[
  a\leq u_1\leq\cdots\leq u_m.
\]
Apply \Cref{lem:null-coordinate-comparison}(b) and repeat the radial argument
above for the vectors $y_i$. This gives
\eqref{eq:past-refined-partial-sum-bound}. The same Cauchy--Schwarz calculation
gives \eqref{eq:past-partial-sum-bound}. Since
$\sum_i n_i y_i=-X$ has positive time coordinate, it is future-pointing.
\end{proof}

\subsection{The reverse triangle inequality}
\label{subsec:reverse-triangle}

The partial-sum theorem will be applied in \Cref{sec:unshifted} to the
Lorentzian lifts associated with two Harder--Narasimhan blocks. It yields
future-pointing timelike vectors $\wt X_E$ and
$-\wt X_F$. Their projections to $\NS(X)_\R$, denoted there by $X_E$ and
$-X_F$, satisfy
\[
  \Gamma-C=X_E+(-X_F).
\]
To combine their individual lower bounds into a bound for $\Gamma-C$, we use the standard reverse triangle inequality \cite[Chapter~5, Corollary~31]{ONeill}, recalled here for convenience.

\begin{lemma}
If $u$ and $v$ are future-pointing timelike vectors in a Lorentzian space of
signature $(1,N)$, then
\[
  \sqrt{(u+v)^2}\geq\sqrt{u^2}+\sqrt{v^2}.
\]
\end{lemma}

\begin{proof}
After an orthochronous Lorentz transformation, write
$u=(\sqrt{u^2},0,\dots,0)$. Since $v$ is future-pointing,
\[
  u\cdot v=\sqrt{u^2}\,v_0
  \geq\sqrt{u^2}\sqrt{v^2}.
\]
Therefore
\[
  (u+v)^2=u^2+v^2+2u\cdot v
  \geq\left(\sqrt{u^2}+\sqrt{v^2}\right)^2.
\]
\end{proof}

\section{The unshifted chamber}\label{sec:unshifted}
We now complete the proof of \Cref{thm:main}(a). The key step is
\Cref{prop:minimal-rank-one}, which shows that a weak or strict destabilizer
of minimal positive rank has rank one. We then extract from it a natural
divisorial subobject $\cO_X(-C)\subset\cO_X$ and prove that the resulting
non-zero effective divisor satisfies $C^2<0$.
\begin{proposition}[Minimal destabilizers have rank one]
\label{prop:minimal-rank-one}
Assume $B\cdot\omega<0$. Among the non-zero proper subobjects
$E\subset\cO_X$ in $\cA_{B,\omega}$ satisfying
\[
  \Arg Z_{B,\omega}(E)\geq\Arg Z_{B,\omega}(\cO_X),
\]
every subobject of minimal positive rank has rank one. The same conclusion
holds with the strict phase inequality.
\end{proposition}

We first record the categorical consequence of rank minimality used in the
proof.

\begin{lemma}
\label{lem:minimal-hn-sign-conditions}
Assume $B\cdot\omega<0$, and let
\[
  0\longrightarrow E\longrightarrow\cO_X\longrightarrow Q\longrightarrow0
\]
be exact in $\cA_{B,\omega}$, with $E$ non-zero and proper. Put
$r:=\rk(E)$ and $F:=H^{-1}(Q)$, and assume $r\geq2$. Let
\[
  0=E_0\subset E_1\subset\cdots\subset E_m=E,
  \qquad
  0=F_0\subset F_1\subset\cdots\subset F_n=F
\]
be the saturated $\mu_\omega$-Harder--Narasimhan filtrations. Suppose that
one of the following two minimality conditions holds:
\begin{itemize}
  \item[\textup{(W)}] $\Lambda(E)\geq0$, while every non-zero proper
  subobject $A\subset\cO_X$ with $0<\rk(A)<r$ satisfies $\Lambda(A)<0$.
  \item[\textup{(S)}] $\Lambda(E)>0$, while every non-zero proper
  subobject $A\subset\cO_X$ with $0<\rk(A)<r$ satisfies $\Lambda(A)\leq0$.
\end{itemize}
Then, in case \textup{(W)},
\[
  \Lambda(E_i)<0\quad(1\leq i<m),
  \qquad
  \Lambda(F_j)\geq0\quad(1\leq j\leq n),
\]
whereas in case \textup{(S)},
\[
  \Lambda(E_i)\leq0\quad(1\leq i<m),
  \qquad
  \Lambda(F_j)>0\quad(1\leq j\leq n).
\]
\end{lemma}

\begin{proof}
By \Cref{lem:subobject-structure}, $E$ is a torsion-free sheaf in
$\cT_{B,\omega}$, while $F$ is a torsion-free sheaf in
$\cF_{B,\omega}$ of rank $r-1>0$.

For $1\leq i<m$, both $E_i$ and $E/E_i$ belong to
$\cT_{B,\omega}$. Thus
\[
  0\longrightarrow E_i\longrightarrow E\longrightarrow E/E_i
  \longrightarrow0
\]
is exact in the tilted heart by \Cref{lem:tilt-exactness}, and composition
with $E\hookrightarrow\cO_X$ makes $E_i$ a non-zero subobject of $\cO_X$.
It is proper: if $E_i\to\cO_X$ were an isomorphism, then
$E\to\cO_X$ would have a right inverse and hence would itself be an
isomorphism in $\cA_{B,\omega}$. Since $0<\rk(E_i)<r$, minimality gives
the asserted inequalities for $\Lambda(E_i)$.

For $1\leq j\leq n$, the inclusion
$F_j\subset F=\ker_{\Coh}(E\to\cO_X)$ gives a factorization
\[
  \bar f_j:E/F_j\longrightarrow\cO_X.
\]
The source $E/F_j$ lies in $\cT_{B,\omega}$ because it is a quotient of
$E$. The coherent kernel of $\bar f_j$ is $F/F_j\in\cF_{B,\omega}$:
its Harder--Narasimhan factors form a tail of those of $F$ and therefore
have slope at most $B\cdot\omega$. The coherent cokernel is $H^0(Q)$, a
torsion sheaf in $\cT_{B,\omega}$. Hence \Cref{lem:tilt-exactness} shows
that $\bar f_j$ is a monomorphism in $\cA_{B,\omega}$.

This subobject has positive rank smaller than $r$. If it is proper, either
minimality condition gives $\Lambda(E/F_j)\leq0$. If it is the identity
subobject of $\cO_X$, the same conclusion follows from
$\Lambda(\cO_X)=0$. Additivity in
\[
  0\longrightarrow F_j\longrightarrow E\longrightarrow E/F_j
  \longrightarrow0
\]
therefore gives
\[
  \Lambda(F_j)
  =\Lambda(E)-\Lambda(E/F_j)
  \geq\Lambda(E)\geq0.
\]
In case \textup{(S)}, the right-hand side is strictly positive. This proves
the remaining assertions.
\end{proof}

\begin{proof}[Proof of \Cref{prop:minimal-rank-one}]
We prove the weak and strict versions simultaneously. Fix one of the two
cases. If no corresponding subobject exists, the assertion is vacuous. Otherwise,
\Cref{lem:subobject-structure}(i) shows that every non-zero subobject of
$\cO_X$ is a torsion-free sheaf of positive rank. We may therefore choose one
of minimal rank and denote it by $E\subset\cO_X$. In the weak case it satisfies
\[
  \Arg Z_{B,\omega}(E)\geq\Arg Z_{B,\omega}(\cO_X),
\]
whereas in the strict case it satisfies the strict inequality. By
\eqref{eq:phase-lambda-equivalence} and
\eqref{eq:strict-phase-lambda-equivalence}, these conditions are
respectively
\begin{equation}\label{eq:Lambda-E-nonnegative}
  \Lambda(E)\geq0,
  \qquad\text{or}\qquad
  \Lambda(E)>0.
\end{equation}
Set
\[
  r:=\rk(E).
\]
If $r=1$, there is nothing to prove. Assume for contradiction that
$r\geq2$.
By \Cref{lem:subobject-structure}, if $Q$ is the quotient of $\cO_X$ by $E$
in the tilted heart and $F:=H^{-1}(Q)$, then
\begin{equation}\label{eq:E-F-IZC}
  0\longrightarrow F\longrightarrow E
  \longrightarrow I_Z(-C)\longrightarrow0
\end{equation}
is exact in $\Coh(X)$ for an effective divisor $C$ and a zero-dimensional
subscheme $Z$, with
\begin{equation*}
  \rk(F)=r-1.
\end{equation*}
Let
\begin{equation*}
  0=E_0\subset E_1\subset\cdots\subset E_m=E
\end{equation*}
be the $\mu_\omega$-Harder--Narasimhan filtration of $E$, and put
\[
  G_i:=E_i/E_{i-1},
  \qquad r_i:=\rk(G_i).
\]
All $G_i$ have slopes strictly larger than $B\cdot\omega$, and their slopes
are strictly decreasing.

The choice of $E$ satisfies condition \textup{(W)} of
\Cref{lem:minimal-hn-sign-conditions} in the weak case and condition
\textup{(S)} in the strict case. Hence
\begin{equation}\label{eq:Lambda-E-filtration-nonpositive}
  \Lambda(E_i)\leq0
  \qquad (1\leq i<m),
\end{equation}
with strict inequality in the weak case.

Apply the future version of \Cref{thm:ordered-lorentzian-partial-sum} in the
augmented Lorentzian space $\wt V$ to the vectors $\wt U_{G_i}$ with weights
$r_i$. Their time coordinates are decreasing, and by
\eqref{eq:time-threshold} we have
$e\cdot\wt U_{G_m}>\tau>-a$. By
\eqref{eq:augmented-square},
\begin{align*}
  \sum_{i=1}^j r_i(\wt U_{G_i}^2-a^2)
  &=2\sum_{i=1}^j\Lambda(G_i)\\
  &=2\Lambda(E_j).
\end{align*}
Hence the proper partial sums are non-positive by
\eqref{eq:Lambda-E-filtration-nonpositive}, while the total sum is
non-negative by \eqref{eq:Lambda-E-nonnegative}. We conclude that
\[
  \wt X_E:=\sum_i r_i\wt U_{G_i}
\]
is future-pointing and
\[
  \wt X_E^2\geq r^2a^2.
\]
Its projection to $\NS(X)_\R$ is
\begin{equation*}
  X_E:=\sum_i r_iU_{G_i}=r\Gamma+c_1(E).
\end{equation*}
By \eqref{eq:projection-increases-square},
\begin{equation}\label{eq:XE-bound}
  X_E\text{ is future-pointing},
  \qquad X_E^2\geq r^2a^2.
\end{equation}
Since $r\geq2$, the sheaf $F$ has positive rank. Let
\begin{equation*}
  0=F_0\subset F_1\subset\cdots\subset F_n=F
\end{equation*}
be its $\mu_\omega$-Harder--Narasimhan filtration, with factors
\[
  H_j:=F_j/F_{j-1},
  \qquad s_j:=\rk(H_j).
\]
Their slopes are decreasing and are all at most $B\cdot\omega$.

The other half of \Cref{lem:minimal-hn-sign-conditions} gives the following
sign condition along the Harder--Narasimhan filtration steps of $F$:
\begin{equation}\label{eq:Lambda-F-filtration-nonnegative}
  \Lambda(F_j)\geq0
  \qquad (1\leq j\leq n),
\end{equation}
with strict inequalities in the strict case.

Apply the past version of \Cref{thm:ordered-lorentzian-partial-sum} to the
vectors $\wt U_{H_j}$ with weights $s_j$. Their time coordinates are decreasing, and by
\eqref{eq:time-threshold} we have
$e \cdot \wt U_{H_1}\leq\tau<a$. Moreover,
\[
  \sum_{j=1}^k s_j(\wt U_{H_j}^2-a^2)
  =2\Lambda(F_k)\geq0
\]
by \eqref{eq:Lambda-F-filtration-nonnegative}. Hence
\[
  \wt X_F:=\sum_j s_j\wt U_{H_j}
\]
is past-pointing and
\[
  \wt X_F^2\geq(r-1)^2a^2.
\]
Its projection is
\begin{equation*}
  X_F:=\sum_j s_jU_{H_j}=(r-1)\Gamma+c_1(F),
\end{equation*}
so
\begin{equation}\label{eq:XF-bound}
  -X_F\text{ is future-pointing},
  \qquad X_F^2\geq(r-1)^2a^2.
\end{equation}
Taking first Chern classes in \eqref{eq:E-F-IZC} gives
\begin{equation*}
  c_1(E)=c_1(F)-C.
\end{equation*}
Therefore,
\begin{equation*}
  \Gamma-C
  =X_E-X_F
  =X_E+(-X_F).
\end{equation*}
Both vectors on the right are future-pointing and timelike by
\eqref{eq:XE-bound} and \eqref{eq:XF-bound}. The reverse triangle
inequality gives
\begin{align*}
  \sqrt{(\Gamma-C)^2}
  &\geq\sqrt{X_E^2}+\sqrt{X_F^2}\\
  &\geq ra+(r-1)a=(2r-1)a.
\end{align*}
Thus
\begin{equation*}
  (\Gamma-C)^2\geq(2r-1)^2\Gamma^2.
\end{equation*}
Since $r\geq2$, \eqref{eq:Lambda-line-divisor} yields
\begin{equation}\label{eq:rank-one-strict-Lambda}
  \Lambda(\cO_X(-C))
  =\frac{(\Gamma-C)^2-\Gamma^2}{2}
  \geq2r(r-1)\Gamma^2>0.
\end{equation}
In particular, $C\neq0$.

We also have $I_Z(-C)\in\cT_{B,\omega}$ by
\Cref{lem:subobject-structure}. The coherent exact sequence
\[
  0\longrightarrow I_Z(-C)\longrightarrow\cO_X(-C)
  \longrightarrow\cO_Z(-C)\longrightarrow0
\]
shows that $\cO_X(-C)\in\cT_{B,\omega}$, because both $I_Z(-C)$ and the
zero-dimensional sheaf $\cO_Z(-C)$ lie in $\cT_{B,\omega}$. Therefore
\[
  0\longrightarrow\cO_X(-C)\longrightarrow\cO_X
  \longrightarrow\cO_C\longrightarrow0
\]
is exact in $\cA_{B,\omega}$. By \eqref{eq:rank-one-strict-Lambda} and
\eqref{eq:strict-phase-lambda-equivalence}, $\cO_X(-C)$ strictly
destabilizes $\cO_X$. It has rank one, contradicting the minimality of
$r\geq2$. This contradiction proves that $r=1$.
\end{proof}

\subsection{Extraction of the destabilizer}

Let $E$ be a minimal-rank destabilizer in either the weak or the strict
case. By \Cref{prop:minimal-rank-one}, $\rk(E)=1$. Then $F=0$ in
\eqref{eq:E-F-IZC}, and
\[
  E=I_Z(-C).
\]
If $C=0$, then \eqref{eq:Lambda-ideal-divisor} gives
\[
  \Lambda(E)=-\ell(Z).
\]
This is negative if $Z\neq\emptyset$. If $Z=\emptyset$, then
$E\cong\cO_X$. The inclusion $E\hookrightarrow\cO_X$ is then a non-zero
endomorphism of $\cO_X$. Since $X$ is connected and projective,
$H^0(X,\cO_X)=\C$, so this endomorphism is an isomorphism, contradicting
properness. Hence $C\neq0$.

By \eqref{eq:Lambda-ideal-divisor},
\begin{equation}\label{eq:rank-one-O-C-nonnegative}
  \Lambda(\cO_X(-C))
  =\Lambda(I_Z(-C))+\ell(Z)
  \geq0.
\end{equation}
In the strict case, the right-hand side is strictly positive.
As above, $I_Z(-C)\in\cT_{B,\omega}$ implies
$\cO_X(-C)\in\cT_{B,\omega}$, so $\cO_X(-C)$ is a subobject of $\cO_X$
in the tilted heart. It remains to prove $C^2<0$.

\begin{lemma}
\label{lem:negative-square}
Let $C\neq0$ be an effective divisor. If
\[
  \cO_X(-C)\in\cT_{B,\omega}
  \qquad\text{and}\qquad
  \Lambda(\cO_X(-C))\geq0,
\]
then $C^2<0$.
\end{lemma}

\begin{proof}
Put
\[
  W:=\Gamma-C.
\]
By \eqref{eq:Lambda-line-divisor},
\begin{equation}\label{eq:W-square}
  W^2\geq\Gamma^2=a^2.
\end{equation}
Since $\cO_X(-C)\in\cT_{B,\omega}$,
\[
  -C\cdot\omega>B\cdot\omega=-p,
\]
so $p-C\cdot\omega>0$. Therefore
\begin{equation*}
  e\cdot W
  =\frac{\Gamma\cdot\omega-C\cdot\omega}{\sqrt h}
  =\tau+\frac{p-C\cdot\omega}{\sqrt h}
  >\tau.
\end{equation*}
Because $W^2\geq a^2$ and $e\cdot W>\tau>-a$, the vector $W$ is
future-pointing: a past-pointing vector of square at least $a^2$ has time
coordinate at most $-a$. The vector $\Gamma$ is also future-pointing, since
$\Gamma^2=a^2$ and
\[
  e\cdot\Gamma=\tau+\frac p{\sqrt h}>\tau>-a.
\]

Because $C$ is a non-zero effective divisor and $\omega$ is ample,
\begin{equation*}
  t_\Gamma:=e\cdot\Gamma
  >e\cdot W=:t_W\geq a.
\end{equation*}
Write
\[
  \Gamma=t_\Gamma e+\gamma,
  \qquad W=t_W e+w,
  \qquad \gamma,w\in e^\perp.
\]
Then
\[
  \norm{\gamma}=\sqrt{t_\Gamma^2-a^2},
\]
while \eqref{eq:W-square} implies
\[
  \norm{w}
  =\sqrt{t_W^2-W^2}
  \leq\sqrt{t_W^2-a^2}.
\]
For $t>s\geq a$, the function $f(u)=\sqrt{u^2-a^2}$ satisfies
\begin{equation*}
  f(t)-f(s)
  =(t-s)\frac{t+s}{f(t)+f(s)}
  >t-s,
\end{equation*}
because $f(u)<u$ for $u\geq a$. Applying this with $t=t_\Gamma$ and
$s=t_W$, we obtain
\begin{align*}
  \norm{\gamma-w}
  &\geq\norm{\gamma}-\norm{w}\\
  &\geq f(t_\Gamma)-f(t_W)\\
  &>t_\Gamma-t_W.
\end{align*}
Finally,
\begin{align*}
  C^2
  &=(\Gamma-W)^2\\
  &=(t_\Gamma-t_W)^2-\norm{\gamma-w}^2\\
  &<0.
\end{align*}
\end{proof}

Applying \Cref{lem:negative-square} to
\eqref{eq:rank-one-O-C-nonnegative} proves that $C^2<0$. In the weak case
the resulting divisorial subobject has phase at least that of $\cO_X$. In
the strict case it has strictly larger phase. We have therefore concluded
both non-trivial implications in \Cref{thm:main}(a) for $L=\cO_X$. Tensor
reduction proves them for every line bundle.

\section{The shifted chamber}\label{sec:shifted}

By tensor reduction, it is again enough to treat $L=\cO_X$. Thus assume
\[
  \cO_X[1]\in\cA_{B,\omega},
  \qquad\text{equivalently}\qquad B\cdot\omega\geq0.
\]

We treat the boundary and positive chambers separately. On
$B\cdot\omega=0$, \Cref{lem:boundary-stability} shows directly that
$\cO_X[1]$ has no non-zero proper subobjects. For $B\cdot\omega>0$, derived
duality reduces the required weak and strict detection statements to the
unshifted result at the parameter $-B$, and transforms the short exact
sequence associated with the resulting effective divisor $C$ satisfying
$C^2<0$ into one whose subobject is
$\cO_X(C)|_C\subset\cO_X[1]$. Together these arguments prove the two
non-trivial implications in \Cref{thm:main}(b).

\subsection{The boundary wall \texorpdfstring{$B\cdot\omega=0$}{B dot omega = 0}}

\begin{lemma}
\label{lem:boundary-stability}
If $B\cdot\omega=0$, then $\cO_X[1]$ has no non-zero proper subobject in
$\cA_{B,\omega}$. In particular, it is $\sigma_{B,\omega}$-stable.
\end{lemma}

\begin{proof}
This is \cite[Lemma~6.1]{AM}. We include the argument to make the boundary
case and our conventions explicit.
Suppose
\begin{equation}\label{eq:subobject-O1}
  0\longrightarrow E\longrightarrow\cO_X[1]
  \longrightarrow Q_0\longrightarrow0
\end{equation}
is exact in $\cA_{B,\omega}$ and $E$ is proper. Coherent cohomology gives
\begin{equation}\label{eq:cohomology-O1}
  0\longrightarrow H^{-1}(E)\longrightarrow\cO_X
  \longrightarrow H^{-1}(Q_0)\longrightarrow H^0(E)
  \longrightarrow0,
\end{equation}
and $H^0(Q_0)=0$.

The image of $\cO_X\to H^{-1}(Q_0)$ is a subsheaf of the torsion-free
sheaf $H^{-1}(Q_0)\in\cF_{B,\omega}$. If $H^{-1}(E)$ were a non-zero
proper subsheaf of $\cO_X$, then $\cO_X/H^{-1}(E)$ would be a non-zero
torsion sheaf, which cannot embed into $H^{-1}(Q_0)$. Hence either
$H^{-1}(E)=0$ or $H^{-1}(E)=\cO_X$. The second possibility forces
$H^{-1}(Q_0)\cong H^0(E)$, but the former lies in $\cF_{B,\omega}$ and
the latter in $\cT_{B,\omega}$, so both vanish and $E=\cO_X[1]$, contrary
to properness. Thus $H^{-1}(E)=0$.

Consequently, $E=H^0(E)$ is a sheaf in $\cT_{B,\omega}$, $Q_0=Q[1]$ for
a torsion-free sheaf $Q\in\cF_{B,\omega}$, and
\eqref{eq:cohomology-O1} becomes
\begin{equation}\label{eq:O-Q-E}
  0\longrightarrow\cO_X\longrightarrow Q\longrightarrow E\longrightarrow0.
\end{equation}
Since $B\cdot\omega=0$, \eqref{eq:central-charge-expanded} gives
\[
  \im Z_{B,\omega}(\cO_X[1])=0.
\]
By \eqref{eq:stability-function-range} and additivity in
\eqref{eq:subobject-O1},
\[
  \im Z_{B,\omega}(E)=0.
\]
A non-zero sheaf in $\cT_{B,\omega}$ has strictly positive imaginary part
unless it is supported in dimension zero: a positive-rank quotient has all
slope-HN factors of slope $>B\cdot\omega=0$, and a one-dimensional torsion
sheaf has positive intersection of its effective first Chern class with the
ample class $\omega$. Thus $E$ is zero-dimensional.

For every closed point $x\in X$, the local ring $\cO_{X,x}$ is a regular
local ring of dimension two. A finite-length $\cO_{X,x}$-module has grade
two, equivalently,
\[
  \operatorname{Ext}^{i}_{\cO_{X,x}}(E_x,\cO_{X,x})=0
  \qquad(i<2).
\]
Thus the sheaves $\mathcal{H}\!om(E,\cO_X)$ and
$\mathcal{E}\!xt^1(E,\cO_X)$ vanish. The local-to-global Ext spectral
sequence \cite[(3.16), p.~85]{Huybrechts} then gives
\[
  \Ext_X^1(E,\cO_X)=0.
\]
Hence \eqref{eq:O-Q-E} splits. Then $Q\cong\cO_X\oplus E$ contains
torsion, contradicting $Q\in\cF_{B,\omega}$ unless $E=0$. Therefore there
is no non-zero proper subobject.
\end{proof}

\subsection{The positive chamber and derived duality}

Assume now that
\begin{equation*}
  B\cdot\omega>0.
\end{equation*}
Consider the involutive exact anti-autoequivalence
\begin{equation*}
  \mathbb D(E):=\RSheafHom(E,\cO_X)[1].
\end{equation*}
Because every object of $\Db(X)$ is perfect, biduality gives a canonical
isomorphism $\mathbb D^2\cong\operatorname{id}$. If
$\ch(E)=(r,c,d)$, then
\[
  \ch(\mathbb D E)=(-r,c,-d),
\]
and \eqref{eq:central-charge-expanded} gives the central-charge identity
\begin{equation}\label{eq:duality-central-charge}
  Z_{B,\omega}(\mathbb D E)
  =-\ol{Z_{-B,\omega}(E)}.
\end{equation}
Thus, for an object whose normalized Bridgeland phase is $\phi\in(0,1)$,
derived duality changes the phase to $1-\phi$.

We use the standard duality lemma for the divisorial stability conditions
considered here: whenever $D\cdot\omega<0$,
\begin{equation}\label{eq:duality-stability-equivalence}
  \cO_X\text{ is $\sigma_{D,\omega}$-(semi)stable}
  \quad\Longleftrightarrow\quad
  \cO_X[1]\text{ is $\sigma_{-D,\omega}$-(semi)stable}.
\end{equation}
Arcara--Miles~\cite[Lemma~6.2]{AM} state this equivalence exactly in the
form used here, for both stability and semistability, under the displayed
assumption $D\cdot\omega<0$. We use only
\eqref{eq:duality-stability-equivalence}, with $D=-B$.

Fix either the weak case, in which $\cO_X[1]$ is not stable, or the
strict case, in which it is not semistable. By
\eqref{eq:duality-stability-equivalence}, $\cO_X$ has the corresponding
failure of stability or semistability for $\sigma_{-B,\omega}$. Since
\[
  (-B)\cdot\omega<0,
\]
the two versions of the unshifted theorem give an effective divisor $C$ with
$C^2<0$ such that $\cO_X(-C)$ is a subobject of $\cO_X$ in
$\cA_{-B,\omega}$ and
\begin{equation}\label{eq:dual-starting-phase}
  \Arg Z_{-B,\omega}(\cO_X(-C))
  \geq
  \Arg Z_{-B,\omega}(\cO_X),
\end{equation}
with strict inequality in the strict case.

Since $\cO_X(-C)\in\cA_{-B,\omega}$, the definition of the tilted heart gives
\[
  \omega\cdot(B-C)>0,
  \qquad\text{and hence}\qquad C\cdot\omega<B\cdot\omega.
\]
Thus $\cO_X(C)[1]\in\cA_{B,\omega}$. Since $C$ is Cartier, the exact
sequence
\[
  0\longrightarrow\cO_X(-C)\longrightarrow\cO_X
  \longrightarrow\cO_C\longrightarrow0
\]
is a locally free resolution of $\cO_C$. Dualizing it gives
\[
  \RSheafHom(\cO_C,\cO_X)\cong \cO_X(C)|_C[-1],
  \qquad\text{hence}\qquad
  \mathbb D(\cO_C)\cong\cO_X(C)|_C.
\]
Applying $\mathbb D$ to the associated exact triangle therefore gives
\[
  \cO_X(C)|_C\longrightarrow\cO_X[1]
  \longrightarrow\cO_X(C)[1]\longrightarrow\cO_X(C)|_C[1].
\]
The first three terms all lie in $\cA_{B,\omega}$, so this is the short exact
sequence
\begin{equation}\label{eq:shifted-divisor-exact}
  0\longrightarrow\cO_X(C)|_C
  \longrightarrow\cO_X[1]
  \longrightarrow\cO_X(C)[1]
  \longrightarrow0
\end{equation}
in the tilted heart.

By \eqref{eq:duality-central-charge}, the inequality
\eqref{eq:dual-starting-phase} becomes
\begin{equation}\label{eq:quotient-phase-lower}
  \Arg Z_{B,\omega}(\cO_X(C)[1])
  \leq
  \Arg Z_{B,\omega}(\cO_X[1]),
\end{equation}
with strict inequality in the strict case. For non-zero complex numbers
$z_1,z,z_2$ in the semi-closed upper half-plane with $z=z_1+z_2$, the
argument of $z$ lies between the arguments of $z_1$ and $z_2$. Moreover, if
one of the two endpoint inequalities is strict, then the opposite endpoint
inequality is strict as well. Applying this seesaw property to
\eqref{eq:shifted-divisor-exact} and \eqref{eq:quotient-phase-lower} gives
\[
  \Arg Z_{B,\omega}(\cO_X(C)|_C)
  \geq
  \Arg Z_{B,\omega}(\cO_X[1]),
\]
with strict inequality in the strict case. This proves both non-trivial
implications in \Cref{thm:main}(b) for $\cO_X$, and tensor reduction proves
them for every line bundle.

\section{Consequences and numerical criteria}\label{sec:consequences}

We conclude by deriving numerical and geometric consequences of
\Cref{thm:main}. The first subsection converts the two phase
comparisons into the explicit inequalities $\Theta_C^\pm$, yielding
\Cref{cor:numerical-stability}. We then prove
\Cref{cor:actual-wall,cor:no-negative-divisors}, describing strictly
semistable points and the consequence for surfaces with no negative integral
curves. Finally, we study the behavior under
$(B,\omega,L)\mapsto(kB, k\omega ,L^{\otimes k})$ and prove the equivalence
with numerical $B$-twisted dHYM-semistability in
\Cref{cor:all-scalings}.

\subsection{Explicit phase inequalities}

Let $L$ be a line bundle and set
\[
  \alpha:=c_1(L)-B,
  \qquad s:=\alpha\cdot\omega.
\]
For every effective divisor $C$, define
\begin{align}
  \Theta_C^-(\alpha,\omega)
  &:=(C\cdot\omega)(\omega^2-\alpha^2)
    +(2C\cdot\alpha-C^2)s,
  \label{eq:Theta-minus}\\
  \Theta_C^+(\alpha,\omega)
  &:=(C\cdot\omega)(\omega^2-\alpha^2)
    +(2C\cdot\alpha+C^2)s.
  \label{eq:Theta-plus}
\end{align}

\begin{proposition}
\label{prop:phase-determinants}
Let $C$ be a non-zero effective divisor.
\begin{enumerate}[label=\textup{(\alph*)}]
  \item If $s>0$ and $s-C\cdot\omega>0$, then
  \[
    \Arg Z_{B,\omega}(L(-C))
    \geq \Arg Z_{B,\omega}(L)
    \quad\Longleftrightarrow\quad
    \Theta_C^-(\alpha,\omega)\leq0.
  \]

  \item If $s<0$ and $s+C\cdot\omega\leq0$, then
  \[
    \Arg Z_{B,\omega}(L(C)|_C)
    \geq \Arg Z_{B,\omega}(L[1])
    \quad\Longleftrightarrow\quad
    \Theta_C^+(\alpha,\omega)\leq0.
  \]
\end{enumerate}
In both cases the phase inequality is strict if and only if the numerical
inequality is strict.
\end{proposition}

\begin{proof}
Tensoring by $L^{-1}$ replaces $B$ by $B-c_1(L)=-\alpha$ and preserves all
central charges. It is therefore enough to compare the structure sheaf with
$\cO_X(\pm C)$ for the parameter $-\alpha$.

Put $A:=\omega^2-\alpha^2$. In the unshifted case,
\[
  Z_{-\alpha,\omega}(\cO_X)=\frac A2+is
\]
and
\[
  Z_{-\alpha,\omega}(\cO_X(-C))
  =\left(\frac A2+C\cdot\alpha-\frac{C^2}{2}\right)
    +i(s-C\cdot\omega).
\]
Both numbers lie in the upper half-plane under the hypotheses of (a), and a
direct calculation gives
\begin{align*}
  &\im\!\left(
    Z_{-\alpha,\omega}(\cO_X(-C))
    \ol{Z_{-\alpha,\omega}(\cO_X)}
  \right)=-\frac12\Theta_C^-(\alpha,\omega).
\end{align*}
The sign of this determinant is the sign of the difference of the two
arguments, proving (a).

For (b), \Cref{lem:divisorial-subobject-criterion} gives a short exact
sequence
\[
  0\longrightarrow L(C)|_C\longrightarrow L[1]
  \longrightarrow L(C)[1]\longrightarrow0.
\]
By the seesaw property, the first term has phase at least that of the middle
term if and only if the quotient has phase at most that of the middle term.
After tensor reduction,
\[
  Z_{-\alpha,\omega}(\cO_X(C))
  =\left(\frac A2-C\cdot\alpha-\frac{C^2}{2}\right)
    +i(s+C\cdot\omega),
\]
and
\begin{align*}
  &\im\!\left(
    Z_{-\alpha,\omega}(\cO_X(C))
    \ol{Z_{-\alpha,\omega}(\cO_X)}
  \right)=\frac12\Theta_C^+(\alpha,\omega).
\end{align*}
Multiplying both central charges by $-1$, as required after shifting, does
not change this determinant. Thus the quotient has phase at most that of
$L[1]$ exactly when $\Theta_C^+\leq0$. The same determinant computations
show the assertions about strict inequalities.
\end{proof}

Combining \Cref{prop:phase-determinants} with \Cref{thm:main} gives a
completely numerical criterion.

\begin{corollary}[Numerical stability criterion]\label{cor:numerical-stability}
Let $\alpha=c_1(L)-B$ and $s=\alpha\cdot\omega$.
\begin{enumerate}[label=\textup{(\alph*)}]
  \item If $s>0$, then $L$ is $\sigma_{B,\omega}$-stable if and only if
  \[
    \Theta_C^-(\alpha,\omega)>0
  \]
  for every non-zero effective divisor $C$ satisfying
  \[
    C^2<0,
    \qquad C\cdot\omega<s.
  \]

  \item If $s<0$, then $L[1]$ is $\sigma_{B,\omega}$-stable if and only if
  \[
    \Theta_C^+(\alpha,\omega)>0
  \]
  for every non-zero effective divisor $C$ satisfying
  \[
    C^2<0,
    \qquad C\cdot\omega\leq -s.
  \]

  \item If $s=0$, then $L[1]$ is $\sigma_{B,\omega}$-stable.
\end{enumerate}
In (a) and (b), semistability of the same object is characterized by the
same conditions with $>0$ replaced by $\geq0$. The side conditions on $C$
are unchanged.
\end{corollary}

\begin{proof}
We prove the stability and semistability statements in parallel. Suppose first
that $s>0$. By \Cref{lem:divisorial-subobject-criterion}, for a non-zero
effective divisor $C$ the morphism
\[
  L(-C)\longrightarrow L
\]
is a subobject in $\cA_{B,\omega}$ exactly when $C\cdot\omega<s$. Thus
\Cref{thm:main}(a) says that $L$ fails to be stable, respectively semistable,
if and only if there is such a divisor with $C^2<0$ for which
\[
  \Arg Z_{B,\omega}(L(-C))
  \geq \Arg Z_{B,\omega}(L),
\]
respectively
\[
  \Arg Z_{B,\omega}(L(-C))
  > \Arg Z_{B,\omega}(L).
\]
By \Cref{prop:phase-determinants}(a), these are equivalent to
$\Theta_C^-(\alpha,\omega)\leq0$, respectively
$\Theta_C^-(\alpha,\omega)<0$. Negating these conditions gives exactly the
strict inequality in the stability criterion and the weak inequality in the
semistability criterion.

Now suppose that $s<0$. By
\Cref{lem:divisorial-subobject-criterion}, the morphism
\[
  L(C)|_C\longrightarrow L[1]
\]
is a subobject exactly when $C\cdot\omega\leq-s$. Applying
\Cref{thm:main}(b) and then \Cref{prop:phase-determinants}(b) in the same way,
failure of stability, respectively semistability, is equivalent to the
existence of such a divisor with $C^2<0$ and
\[
  \Theta_C^+(\alpha,\omega)\leq0,
  \qquad\text{respectively}\qquad
  \Theta_C^+(\alpha,\omega)<0.
\]
Negating again proves (b), together with its semistable analogue.

Finally, if $s=0$, tensoring by $L^{-1}$ reduces to
\Cref{lem:boundary-stability}, which gives the stability of $L[1]$.
\end{proof}

\subsection{Strictly semistable points}

\begin{proof}[Proof of \Cref{cor:actual-wall}]
Put $\theta:=\Arg Z(T)$. Since $T$ is not stable,
\Cref{thm:main} gives a negative effective divisor $C$ and a divisorial
subobject $D_C\subset T$ with
\[
  \Arg Z(D_C)\geq\theta.
\]
Semistability of $T$ forces equality. By
\Cref{lem:divisorial-subobject-criterion}, the quotient is
$Q_C\cong L|_C$ when $T=L$ and $Q_C\cong L(C)[1]$ when $T=L[1]$. In
particular, $Q_C\neq0$. Additivity gives
\[
  Z(Q_C)=Z(T)-Z(D_C)\in\R e^{i\theta}.
\]
The coefficient is positive. Indeed, $Z(Q_C)$ is non-zero and lies in the
semi-closed upper half-plane by \eqref{eq:stability-function-range}.
A negative multiple of $e^{i\theta}$ lies in the lower half-plane when
$0<\theta<\pi$, while for $\theta=\pi$ it lies on the forbidden positive
real axis. Hence $Z(Q_C)$ is a positive real multiple of $e^{i\theta}$,
and $Q_C$ also has argument $\theta$.

The object $D_C$ is semistable. Indeed, any subobject of $D_C$ with argument
strictly larger than $\theta$ would, after composition with $D_C\subset T$,
be a subobject of $T$ with argument strictly larger than $\theta$.

It remains to prove that $Q_C$ is semistable. Suppose that
$A\subset Q_C$ has argument strictly larger than $\theta$, and let
$\widetilde A\subset T$ be its inverse image. Then
\[
  0\longrightarrow D_C\longrightarrow\widetilde A
  \longrightarrow A\longrightarrow0
\]
is exact. The central charge of $D_C$ lies on the ray of angle $\theta$, while
that of $A$ has larger argument. Therefore the argument of
$Z(\widetilde A)=Z(D_C)+Z(A)$ is strictly larger than $\theta$, contradicting
the semistability of $T$. Thus $Q_C$ is semistable, and all three terms have
the same argument.
\end{proof}

\begin{proof}[Proof of \Cref{cor:no-negative-divisors}]
First note that the hypothesis excludes every non-zero effective divisor of
negative self-intersection. Indeed, if
\[
  C=\sum_i m_iC_i
\]
is the decomposition of an effective divisor into distinct integral
components, then $C_i\cdot C_j\geq0$ for $i\neq j$. Hence
\[
  C^2=\sum_i m_i^2C_i^2
      +2\sum_{i<j}m_im_j(C_i\cdot C_j)\geq0
\]
whenever every $C_i^2\geq0$.

For a fixed line bundle $L$, either $L$ or $L[1]$ belongs to the tilted
heart. If that object were not stable, \Cref{thm:main} would produce a
non-zero effective divisor of negative self-intersection, contrary to the
preceding observation. Hence the appropriate shift is stable. Stability
is preserved under shifts, so $L$ is stable as an object of $\Db(X)$.
\end{proof}

\subsection{Stability under all integral scalings}

The scaling argument below follows \cite[Theorem~2.9]{Fan}, building on the
large-scaling viewpoint introduced in
\cite[Proposition~6.2]{Stoppa}. The new input is
\Cref{cor:numerical-stability}, which removes the use of
\cite[Conjecture~2.8]{Fan} in the implication from numerical twisted dHYM
semistability to stability under all integral scalings.

\begin{proof}[Proof of \Cref{cor:all-scalings}]
For an effective divisor
$C=\sum_i m_iC_i$ with integral components, the expression
\eqref{eq:Psi-dHYM-intro} is linear in $C$, so numerical twisted dHYM
semistability implies
\begin{equation}\label{eq:Psi-effective}
  \Psi_C(\alpha,\omega)\geq0
\end{equation}
for every effective divisor $C$.

For $k\geq1$, put
\[
  L_k:=L^{\otimes k},
  \qquad B_k:=kB,
  \qquad \omega_k:=k\omega.
\]
The relative class and its slope are
\[
  c_1(L_k)-B_k=k\alpha,
  \qquad
  (k\alpha)\cdot(k\omega)=k^2s.
\]
A direct substitution in \eqref{eq:Theta-minus} and
\eqref{eq:Theta-plus} gives
\begin{align}
  \Theta_C^-(k\alpha,k\omega)
  &=k^2\bigl(k\Psi_C(\alpha,\omega)-C^2s\bigr),
  \label{eq:scaled-Theta-minus}\\
  \Theta_C^+(k\alpha,k\omega)
  &=k^2\bigl(k\Psi_C(\alpha,\omega)+C^2s\bigr).
  \label{eq:scaled-Theta-plus}
\end{align}

Assume first that $L$ is numerically $B$-twisted dHYM-semistable. If
$s>0$ and $C^2<0$, then \eqref{eq:Psi-effective} and
\eqref{eq:scaled-Theta-minus} give
\[
  \Theta_C^-(k\alpha,k\omega)>0
\]
for every $k\geq1$. Hence \Cref{cor:numerical-stability}(a) shows that
$L_k$ is $\sigma_{kB,k\omega}$-stable. If $s<0$ and $C^2<0$, then
$C^2s>0$, and
\eqref{eq:scaled-Theta-plus} gives
\[
  \Theta_C^+(k\alpha,k\omega)>0.
\]
Now \Cref{cor:numerical-stability}(b) shows that $L_k[1]$, and therefore
$L_k$, is stable. If $s=0$, stability follows from
\Cref{cor:numerical-stability}(c). This proves (i)$\Rightarrow$(ii), while
(ii)$\Rightarrow$(iii) is immediate.

Assume conversely that $L_k$ is $\sigma_{kB,k\omega}$-stable for all
sufficiently large $k$, and fix an integral curve $C$. If $s>0$, then for
all sufficiently large $k$,
\[
  k^2s-k(C\cdot\omega)>0,
\]
so $L_k(-C)\to L_k$ is a subobject in the tilted heart by
\Cref{lem:divisorial-subobject-criterion}. Stability and
\Cref{prop:phase-determinants} imply
\[
  \Theta_C^-(k\alpha,k\omega)>0.
\]
Using \eqref{eq:scaled-Theta-minus} and dividing by $k^3$ gives
\[
  \Psi_C(\alpha,\omega)-\frac{C^2s}{k}>0.
\]
Letting $k\to\infty$ yields $\Psi_C(\alpha,\omega)\geq0$.

If $s<0$, then for all sufficiently large $k$,
\[
  k^2s+k(C\cdot\omega)\leq0,
\]
so $L_k(C)|_C\to L_k[1]$ is a subobject. Stability and
\Cref{prop:phase-determinants} give
\[
  \Theta_C^+(k\alpha,k\omega)>0.
\]
By \eqref{eq:scaled-Theta-plus},
\[
  \Psi_C(\alpha,\omega)+\frac{C^2s}{k}>0,
\]
and again the limit gives $\Psi_C(\alpha,\omega)\geq0$.

Finally, if $s=0$, the Hodge index theorem gives $\alpha^2\leq0$, and hence
\[
  \Psi_C(\alpha,\omega)
  =(C\cdot\omega)(\omega^2-\alpha^2)>0
\]
for every integral curve $C$. Thus (iii)$\Rightarrow$(i), completing the
proof.
\end{proof}


\begin{thebibliography}{99}

\bibitem{AB}
D.~Arcara and A.~Bertram,
\emph{Bridgeland-stable moduli spaces for $K$-trivial surfaces},
J. Eur. Math. Soc. (JEMS) \textbf{15} (2013), no.~1, 1--38,
with an appendix by M.~Lieblich.

\bibitem{AM}
D.~Arcara and E.~Miles,
\emph{Bridgeland stability of line bundles on surfaces},
J. Pure Appl. Algebra \textbf{220} (2016), no.~4, 1655--1677.

\bibitem{Bridgeland}
T.~Bridgeland,
\emph{Stability conditions on triangulated categories},
Ann.\ of Math. (2) \textbf{166} (2007), no.~2, 317--345.

\bibitem{BriK3}
T.~Bridgeland,
\emph{Stability conditions on $K3$ surfaces},
Duke Math. J. \textbf{141} (2008), no.~2, 241--291.

\bibitem{Chen}
G.~Chen,
\emph{The $J$-equation and the supercritical deformed
Hermitian--Yang--Mills equation},
Invent. Math. \textbf{225} (2021), no.~2, 529--602.

\bibitem{CJY}
T.~C.~Collins, A.~Jacob, and S.-T.~Yau,
\emph{$(1,1)$ forms with specified Lagrangian phase: a priori estimates and
algebraic obstructions},
Camb. J. Math. \textbf{8} (2020), no.~2, 407--452.

\bibitem{CollinsShi}
T.~C.~Collins and Y.~Shi,
\emph{Stability and the deformed Hermitian--Yang--Mills equation},
Surv. Differ. Geom. \textbf{24} (2019), 1--38.

\bibitem{Fan}
Y.-W.~Fan,
\emph{Notes on the deformed Hermitian--Yang--Mills equations and the large
scaling limits of stability conditions},
arXiv:2604.22246, 2026.

\bibitem{GKP}
D.~Greb, S.~Kebekus, and T.~Peternell,
\emph{Movable curves and semistable sheaves},
Int. Math. Res. Not. IMRN 2016, no.~2, 536--570.

\bibitem{Hartshorne}
R.~Hartshorne,
\emph{Algebraic geometry},
Graduate Texts in Mathematics \textbf{52},
Springer-Verlag, New York, 1977.

\bibitem{Huybrechts}
D.~Huybrechts,
\emph{Fourier--Mukai transforms in algebraic geometry},
Oxford Mathematical Monographs,
The Clarendon Press, Oxford University Press, Oxford, 2006.

\bibitem{JacobYau}
A.~Jacob and S.-T.~Yau,
\emph{A special Lagrangian type equation for holomorphic line bundles},
Math. Ann. \textbf{369} (2017), no.~1--2, 869--898.

\bibitem{KhalidSjoestroemDyrefelt}
S.~Khalid and Z.~Sj\"ostr\"om Dyrefelt,
\emph{The set of destabilizing curves for deformed Hermitian
Yang--Mills and $Z$-critical equations on surfaces},
Int. Math. Res. Not. IMRN 2024, no.~7, 5773--5814.

\bibitem{LYZ}
N.~C.~Leung, S.-T.~Yau, and E.~Zaslow,
\emph{From special Lagrangian to Hermitian--Yang--Mills via Fourier--Mukai
transform},
Adv. Theor. Math. Phys. \textbf{4} (2000), no.~6, 1319--1341.

\bibitem{MacriSchmidt}
E.~Macrì and B.~Schmidt,
\emph{Lectures on Bridgeland stability},
in \emph{Moduli of curves},
Lect. Notes Unione Mat. Ital. \textbf{21},
Springer, Cham, 2017, 139--211.

\bibitem{MY}
Y.~Mizuno and T.~Yoshida,
\emph{Bridgeland stability of sheaves on del Pezzo surface of Picard rank three},
arXiv:2502.18894, 2025.

\bibitem{ONeill}
B.~O'Neill,
\emph{Semi-Riemannian geometry with applications to relativity},
Pure and Applied Mathematics \textbf{103},
Academic Press, Inc., New York, 1983.

\bibitem{Stoppa}
J.~Stoppa,
\emph{Nakai--Moishezon criteria and the toric Thomas--Yau conjecture},
arXiv:2505.07228, 2025.

\end{thebibliography}
\end{document}